\documentclass[12pt]{amsart}
\usepackage[margin=1in]{geometry}
\usepackage[T1]{fontenc}
\usepackage{enumerate}		
\usepackage{amssymb,amsfonts}
\usepackage{hyperref}
\hypersetup{
    colorlinks=true,
    linkcolor=blue,
    filecolor=magenta,      
    urlcolor=cyan,
    citecolor=magenta
}
\usepackage{cleveref}
\usepackage{amsthm}
\usepackage{amsmath}
\usepackage{graphicx}
\usepackage{lmodern}
\usepackage[english]{babel}

\usepackage[style=alphabetic,sorting=nyt]{biblatex}
\newtheorem{theorem}{Theorem}[section]
\newtheorem*{theorem*}{Theorem}
\newtheorem{cor}[theorem]{Corollary}
\newtheorem*{cor*}{Corollary}
\newtheorem{prop}[theorem]{Proposition}
\newtheorem*{prop*}{Proposition}
\newtheorem{lemma}[theorem]{Lemma}
\newtheorem*{lemma*}{Lemma}

\newtheorem*{quest*}{Question}
\newtheorem{conj}[theorem]{Conjecture}
\newtheorem*{conj*}{Conjecture}

\newtheorem*{fact*}{Fact}

\newtheorem*{claim*}{Claim}

\theoremstyle{definition}
\newtheorem{defn}[theorem]{Definition}
\newtheorem*{defn*}{Definition}

\newtheorem*{exs*}{Examples}
\newtheorem{ex}[theorem]{Example}
\newtheorem*{ex*}{Example}
\newtheorem{notn}[theorem]{Notation}
\newtheorem*{notn*}{Notation}

\newtheorem*{assumption*}{Assumption}

\theoremstyle{remark}
\newtheorem{rmk}[theorem]{Remark}
\newtheorem*{rmk*}{Remark}

\numberwithin{equation}{section}

\DeclareMathOperator{\CC}{\mathbb{C}}
\DeclareMathOperator{\RR}{\mathbb{R}}

\DeclareMathOperator{\QQ}{\mathbb{Q}}
\DeclareMathOperator{\ZZ}{\mathbb{Z}}

\title{Some equations involving the Gamma function}
 \author{Sebastian Eterovi\'c}
 \address{School of Mathematics, University of Leeds, Leeds LS2 9JT, UK} 
 \email{s.eterovic@leeds.ac.uk}

 \author{Adele Padgett}
 \address{Kurt G\"odel Research Center, Universit\"at Wien, 1090 Wien, Austria}
 \email{adele.lee.padgett@univie.ac.at}

\date{December 8, 2023}

\thanks{SE was supported by EPSRC fellowship EP/T018461/1. AP was supported by the Fields Institute.}

\keywords{Gamma function, existential closedness, Rouch\'e's theorem}

\subjclass[2020]{30C15, 33B15, 30D35, 32A60}

\begin{document}

\begin{abstract}
    Let $V\subseteq\CC^{2n}$ be an algebraic variety with no constant coordinates and with a dominant projection onto the first $n$ coordinates. We show that the intersection of $V$ with the graph of the $\Gamma$ function is Zariski dense in $V$. Our method gives an explicit description of the distribution of these intersection points, and can be adapted for some other functions. 
\end{abstract}

\maketitle

\section{Introduction}
The $\Gamma$ function is an important analytic map which, among other things, extends the factorial operation on positive integers to complex numbers by being a solution to the difference equation $f(z+1)=zf(z)$.   
It is usually defined on the right half plane $\{z\in\CC : \Re(z)>0\}$ via the following integral:
\begin{equation*}
    \Gamma(z) := \int_0^{\infty}t^{z-1}\exp(-t)\mathrm{d}t.
\end{equation*}
From this one gets that $\Gamma(1)=1$ and that $\Gamma(n)=(n-1)!$, for every positive integer $n$. 
This function has an analytic continuation to $\mathbb{C}\setminus\mathbb{Z}^{\leq 0}$, having a simple pole at every non-positive integer. 

In this paper we prove the following result. 
\begin{theorem}
\label{thm:ecforgamma}
    Let $n$ be a positive integer and let $V\subseteq\CC^{2n}$ be an algebraic variety with no constant coordinates. Let $\pi_1:\mathbb{C}^{2n}\to\mathbb{C}^n$ denote the projection onto the first $n$ coordinates. If $\dim \pi_1(V) = n$, then $V$ has a Zariski dense set of points of the form $(\mathbf{z},\Gamma(\mathbf{z}))$.
\end{theorem}
This follows a long list of similar results for different important transcendental functions in number theory that have appeared in recent years, such as for the complex exponential function (\cite{expEC,brownawell-masser,dft:ec,mantova-masser,marker}), the exponential maps of (semi-) abelian varieties (\cite{expEC,gallinaro-abelianeqs}), the modular $j$ function (\cite{aem,jEC,gallinaro:raisingtopowers}), and the uniformization maps of Shimura varieties (\cite{eterovic-zhao}). 

The unifying theme behind all these results is that of understanding the algebraic properties of the function at hand. 
To be more precise, given a transcendental holomorphic function $f$, what we aim for is to find a minimal set of geometric conditions an algebraic variety $V$ should satisfy to ensure that the intersection between $V$ and the graph of $f$ is Zariski dense in $V$. 
    By ``geometric conditions'' we mean algebro-geometric conditions that make sense on any algebraically closed field of characteristic zero, like the domination condition required of $V$ in Theorem~\ref{thm:ecforgamma}. 
    In particular, Theorem~\ref{OneVarEC} (the case $n=1$ of Theorem~\ref{thm:ecforgamma}) completely solves the problem for subvarieties of $\CC^2$. 
    However, as in results mentioned earlier, one should expect something weaker than the domination condition to suffice (see Conjecture \ref{conj:ecgamma} for a precise statement).
    
Studying the intersection of algebraic varieties with the graph of transcendental functions is a natural problem, and it is closely related to the questions about value distributions of $f$ with moving targets, a classical problem in complex analysis. 
This question has received different names: it has been called \emph{Zilber's Nullstellensatz} and \emph{exponential algebraic closedness} in the case of complex exponentiation, after the highly influential work of Zilber on pseudoexponentiation (\cite{zilber:expsums,zilberexp}), and more generally it has been named the \emph{existential closedness problem} (EC). 

Our main theorem solves some cases of the problem for the $\Gamma$ function. 
The articles cited above provide different methods for finding  points in the intersection of an algebraic variety $V$ and the graph of the corresponding function $f$. 
We remark that all the functions considered in these articles (exponential maps and automorphic functions) are periodic, and in each case, the methods exploit the periodicity of the function being studied. 
Moreover, many of the methods provide not only the existence of points in the intersection of $V$ with the graph of $f$, but also an indication of how these points are distributed.
Our aim is to develop an approach that works for the $\Gamma$ function (a non-periodic function of great interest) while still providing information about the distribution of intersection points.
For this we will extend some of the methods used in \cite{jEC} and \cite{eterovic-zhao} which rely on Rouch\'e's theorem. 
We will also be aided by its close companion, the Argument Principle.
This will allow us not only to prove Theorem~\ref{thm:ecforgamma}, but also to solve certain systems of analytic equations involving $\Gamma$, see Proposition~\ref{prop:analyticec}. 
For example, this will show that the equation $\Gamma(z) = \exp\left(\frac{1}{z}\right)$ has infinitely many solutions. 
Furthermore, the method will give us an explicit description of the distribution of solutions, see Corollary~\ref{DistributionLowerBd}.

We remark that this venture is motivated not only by solving this type of problem for the $\Gamma$ function, but also by the desire for obtaining a new unifying method that can be adapted to all the known cases. 
Indeed, a similar strategy to the one we present here has been recently used to solve equations involving periodic functions (in one variable) in \cite[\S 4]{aem}.
To illustrate this, we will briefly explain in \S\ref{subsec:discussion} how our method adapts to give a new proof of the analogue of Theorem \ref{thm:ecforgamma} for the complex exponential function, which appears in \cite{brownawell-masser,dft:ec,expEC}.
Furthermore, Theorem~\ref{thm:ecforgamma} has the same level of generality as many of the results of the articles cited, except for \cite{gallinaro-abelianeqs,gallinaro:raisingtopowers,gallinaro-expsums,mantova-masser} which have other kinds of special restrictions. 

The structure of the paper is as follows. In \S\ref{sec:background} we give some background on the $\Gamma$ function, including some transcendence questions and a description of the behavior of $\Gamma$ on the upper right quadrant of $\CC$. 
Then in \S\ref{sec:n=1} we give two proofs of Theorem~\ref{thm:ecforgamma} in the case $n=1$. 
Finally in \S\ref{sec:proof} we prove the main result using some of the methods in \S\ref{sec:n=1}.

\subsection{Notation and conventions}

\begin{enumerate}[(a)]
    \item We will denote tuples of complex numbers using bold face letters. So a tuple $(z_1,\ldots,z_n)\in\CC^n$ will be denoted as $\mathbf{z}$. Furthermore, if $f$ denotes a one-variable function, then we write $f(\mathbf{z})$ to denote the tuple $(f(z_1),\ldots,f(z_n))$.
    \item We will use $\Gamma$ to denote both the function in one complex variable, and also for its extension to $n$-variables $\Gamma:\left(\CC\setminus\ZZ^{\leq 0}\right)^n\to\CC^n$ via $\Gamma(\mathbf{z}) = (\Gamma(z_1),\ldots,\Gamma(z_n))$. Any ambiguities will be avoided by context.
    \item We let $\Re(z)$ and $\Im(z)$ denote the real and imaginary parts of a complex number $z\in \CC$, respectively. 
    \item Given $z\in\CC$ and $\epsilon>0$, we let $B(z,\epsilon)$ denote the open ball centered at $z$ of radius $\epsilon$.
    \item When talking about the argument of a complex number $z\in\CC$, we will always assume that the argument function is defined on all of $\CC^\times$. Of course, this implies that $\arg(z)$ is not continuous on all of $\CC^\times$, but it is continuous when restricting to the domain of a specific branch of logarithm. Unless otherwise stated, we express the argument function as $\arg:\CC^\times\to(-\pi,\pi]$ and if $\Re(z),\Im(z)>0$, then $\arg(z) = \arctan\left(\frac{\Im(z)}{\Re(z)}\right)$. 
    \item Given a subset $X\subset\CC^n$, we let $\overline{X}$ and $X^\circ$ denote the topological closure and interior of $X$, respectively. We also use $\partial X$ to denote the boundary of $X$, that is, $\partial X = \overline{X}\setminus X^\circ$.
\end{enumerate}

\subsection*{Acknowledgements}
We would like to thank Vahagn Aslanyan, Gareth Jones, Jonathan Kirby, and Vincenzo Mantova for constructive discussions around the topics presented here. We would also like to thank the anonymous referees for their useful comments.

\section{Background on \texorpdfstring{$\Gamma$}{Gamma}}
\label{sec:background}
We recall here some well-known facts about $\Gamma$. 
For reference, see e.g.~\cite[Chapter 2]{remmert}. 
As mentioned in the introduction, $\Gamma$ can be defined on the right half plane $\{z\in\CC : \Re(z)>0\}$ via an integral. 
From that definition one also gets that $\Gamma(x)>0$ whenever $x>0$, and $\overline{\Gamma(z)} = \Gamma(\overline{z})$ for every $z$ in the right half plane. 
This function has an analytic continuation to $\mathbb{C}\setminus\mathbb{Z}^{\leq 0}$ which can be described using the following infinite product:
\begin{equation}
\label{eq:gammaweierstrass}
    \Gamma(z)=  \frac{\exp(-\gamma z)}{z}\prod_{k=1}^{\infty}\left(1+\frac{z}{k}\right)^{-1}\exp\left(\frac{z}{k}\right),
\end{equation}
where $\gamma\in\RR^+$ denotes the Euler-Mascheroni constant:\footnote{It is unknown whether $\gamma$ is rational or not.} 
\begin{equation*}
    \gamma = \lim_{n\rightarrow\infty}\left(\sum_{k=1}^{n}\frac{1}{k}-\log(n)\right).
\end{equation*}
$\Gamma$ can also be defined via the following limit
\begin{equation}
\label{eq:gammagausslimit}
    \Gamma(z) = \lim_{n\rightarrow\infty}\frac{n!n^{z}}{z(z+1)\cdots(z+n)}.
\end{equation}
Using (\ref{eq:gammagausslimit}), it is easy to see that whenever $\Re(z)>0$, we have that $|\Gamma(z)|\leq\Gamma(\Re(z))$. 

We note that $\Gamma$ is holomorphic on $\mathbb{C}\setminus\mathbb{Z}^{\leq 0}$ and meromorphic on $\mathbb{C}$, having a simple pole at every point of $\mathbb{Z}^{\leq 0}$.  
It is known that $\Gamma$ does not satisfy any differential equation (more details in the first proof of Theorem~\ref{OneVarEC}), but it does satisfy the following difference equations: 
\begin{enumerate}[(a)]
    \item For every $z\in\mathbb{C}\setminus\mathbb{Z}^{\leq 0}$,
    \begin{equation}
    \label{eq:diffeq1}
        \Gamma(z+1) = z\Gamma(z).
    \end{equation}
    \item For every $z\in\CC$,
    \begin{equation}
    \label{eq:diffeq2}
        \Gamma(z)\Gamma(1-z) = \frac{\pi}{\sin(\pi z)}.
    \end{equation}
    \item For every $z\in\CC\setminus\RR^{\leq0}$ and for every integer $n\geq 2$,
    \begin{equation}
    \label{eq:diffeq3}
        \prod_{k=0}^{n-1}\Gamma\left(z + \frac{k}{n}\right) = (2\pi)^{\frac{1}{2}(n-1)}n^{\frac{1}{2}-nz}\Gamma(nz).
    \end{equation}
\end{enumerate}  
In fact, a theorem of Wielandt shows that $\Gamma$ is the unique function which is holomorphic on $\{z\in\CC: \Re(z)>0\}$, is bounded on $\{z\in\CC: 1\leq\Re(z)< 2\}$, satisfies (\ref{eq:diffeq1}) and $\Gamma(1)=1$ (see \cite[Chapter 2, \S 2.4]{remmert}). 

We also recall Stirling's formula: there is a holomorphic function $\mu(z)$ on $\CC\setminus\RR^{\leq0}$ such that $\lim_{z\to\infty}\mu(z)=0$ and 
\begin{equation}
\label{eq:stirlingeq}
    \Gamma(z) = \sqrt{2\pi}z^{z-\frac{1}{2}}\exp(-z)\exp(\mu(z)).
\end{equation}
There are more precise descriptions of $\mu(z)$, see \cite[Chapter 2, \S 4.2]{remmert}. 
In particular, when $\Re(z), \Im(z) > 0$, there is $M>0$ such that if $|z|>M$ then $|\mu(z)| < 1$, and so $\frac{1}{e}\leq|\exp(\mu(z))|\leq e$, where $e=\exp(1)$. 
Thus, if we write $z=x+iy$ with $x,y>0$, we get that when $|z|> M$
\begin{equation}
\label{eq:stirlingineq}
   \frac{\sqrt{2\pi}}{e}|z|^{x-\frac{1}{2}}\exp\left(- y\arg(z) - x\right) \leq|\Gamma(z)| \leq e\sqrt{2\pi}|z|^{x-\frac{1}{2}}\exp\left(- y\arg(z) - x\right).
\end{equation}
One can use (\ref{eq:stirlingineq}) to obtain that $|\Gamma(x+iy)|$ tends exponentially to zero as $|y|$ tends to $+\infty$. 
More precisely, if $K$  is a compact subset of $\RR$, then there are $C>0$ and $y_0\in\RR$ such that whenever $x\in K$ and $|y|>|y_0|$ we have
 \begin{equation}
 \label{eq:modgammaexpdecrease}
     |\Gamma(x+iy)|\leq C \sqrt{2\pi}|y|^{x-\frac{1}{2}}\exp\left(-\frac{\pi}{2}|y|\right).
 \end{equation}

\subsection{Behavior of \texorpdfstring{$\Gamma$}{Gamma} in the upper right quadrant}

In this section, we always use $z$ to denote a complex number, $x$ its real part, and $y$ its imaginary part. 
We denote the upper right quadrant as $Q:=\{z\in\mathbb{C} : \Re(z),\Im(z) > 0\}$.
Let $\alpha$ denote the unique positive real zero of $\Gamma'$ ($\alpha \approx 1.4616\dots$, see \cite{uchiyama}). 

We first remark that $|\Gamma(z)|$ decreases exponentially to zero along vertical lines (with $y$ tending to $+\infty$) by (\ref{eq:modgammaexpdecrease}), and $|\Gamma(z)|$ grows exponentially along horizontal lines (with $x$ tending to $+\infty$) by (\ref{eq:stirlingineq}). For future reference, we recount this in the following lemmas and corollaries (which are immediate). 

\begin{lemma}\label{GammaStrictlyDecrY}
    For fixed $x>0$, $|\Gamma(x+iy)|$ decreases exponentially as $y\to+\infty$. In particular, the function $y\mapsto |\Gamma(x+iy)|$ is injective and has non-vanishing derivative at every $y>0$. 
\end{lemma}

\begin{lemma}\label{GammaStrictlyIncrX}
    For fixed $y>0$ and $x \ge \alpha$, $|\Gamma(x+iy)|$ increases exponentially as $x\to+\infty$. In particular, the function $x\mapsto |\Gamma(x+iy)|$ is injective and has non-vanishing derivative at every $x>\alpha$. 
\end{lemma}

\begin{cor}\label{GammaLimitYtoInfty}
    Let $A(z)$ be any non-zero algebraic function.
    For fixed $x > 0$, if $A(x+iy)$ is defined for all $y$ large enough, then $\displaystyle \lim_{y \to +\infty} \frac{\Gamma(x+iy)}{A(x+iy)} = 0$.
\end{cor}

\begin{cor}\label{GammaLimitXtoInfty}
    Let $A(z)$ be any non-zero algebraic function.
    For fixed $y$, if $A(x+iy)$ is defined for all large enough $x$, then $\displaystyle\lim_{x \to +\infty}\left|\frac{\Gamma(x+iy)}{A(x+iy)}\right| = \infty$.
\end{cor}

Given positive real numbers $r_1$ and $r_2$, we define
\[
    Q_{r_1,r_2} := \{z \in Q: \Re(z) > r_1,\Im(z)  > r_2\}.
\]
Now we will study the family of sets $C_r := \{z\in\CC : |\Gamma(z)| = r\}$ for $r \in \RR^+$.
We will focus on how these sets behave in $Q_{\alpha,0}$, though analogous statements hold in its reflection over the real axis because $\overline{\Gamma(z)} = \Gamma(\overline{z})$.

\begin{prop}\label{ConstModulusCurves}
    For each $r \in \RR^+$, there is a function $y_r(x)$ such that for all $x>\alpha$, $|\Gamma(x+iy_r(x))|=r$. The graph of this function is contained in $C_r$ and forms a single $C^1$ curve with positive slope and no horizontal or vertical asymptotes.
\end{prop}

\begin{proof}
Let $r \in \RR^+$.
By Corollaries~\ref{GammaLimitYtoInfty} and ~\ref{GammaLimitXtoInfty} and the intermediate value theorem, there are $x_0 > \alpha$ and $y_0>0$ such that $|\Gamma(x_0+iy_0)| = r$.
By Lemma~\ref{GammaStrictlyDecrY}, $\displaystyle \frac{\partial |\Gamma(x+iy)|}{\partial y}(x_0 + iy_0) \ne 0$, so we can use the implicit function theorem to obtain a unique $C^1$ function $y(x)$ on some interval $(x_0-\epsilon,x_0+\epsilon)$ such that $C_r$ matches the graph of $y(x)$ in a neighborhood of $(x_0,y_0)$.
Again using the intermediate value theorem and the fact that $|\Gamma(x+iy)|$ strictly decreases vertically and strictly increases horizontally, the slope of the implicit function $y(x)$ must be positive.

Since $|\Gamma(x+iy)|$ is continuous and $\displaystyle \frac{\partial |\Gamma(x+iy)|}{\partial y}$ is nonzero in the upper right quadrant, we can extend the implicit function $y(x)$ on the left all the way to either $x = \alpha$ or the $x$-axis so that it is still a $C^1$ curve with positive slope satisfying $|\Gamma(x+iy(x))| = r$.
On the right, we can also extend $y(x)$, and we will show that the extension does not approach a vertical or horizontal asymptote.
If $y(x)$ has a vertical asymptote at $x = a$, then $|\Gamma(a+1+ iy)| > r$ for all $y>0$ by Lemma~\ref{GammaStrictlyIncrX}, but that contradicts Corollary~\ref{GammaLimitYtoInfty}.
If $y(x)$ has a horizontal asymptote at $y = b$, then $|\Gamma(x+i(b+1))|<r$ for all $x>0$ by Lemma~\ref{GammaStrictlyDecrY}, but that contradicts Corollary~\ref{GammaLimitXtoInfty}.

Let $K_r,K_r' \subset Q_{\alpha,0} \cap C_r$ denote two irreducible curves. 
$K_r$ and $K_r'$ must be identically equal if they intersect because otherwise the set of $(x,y)$ such that $|\Gamma(x+iy)| = r$ would not locally be the graph of a function.
If $K_r$ and $K_r'$ are disjoint, let $y_*$ be least such that there are $x,x'$ with $(x,y_*) \in K_r$ and $(x',y_*) \in K_r'$.
Then $x \ne x'$ but $|\Gamma(x+iy_*)| = r = |\Gamma(x'+iy_*)|$, which contradicts Lemma~\ref{GammaStrictlyIncrX}.
\end{proof}

\begin{lemma}\label{Limity'toInfty}
    Given $r \in \RR^+$, let $y_r(x)$ denote the function given by Proposition \ref{ConstModulusCurves}. 
    Then $y'_r(x)\geq 2\log\left(\lfloor x\rfloor\right)-2$ when $x>\alpha$. 
    In particular, $\displaystyle\lim_{x \to +\infty}y_r'(x) = +\infty$.
\end{lemma}

\begin{proof}
    First, note that $\frac{d}{dx}(|\Gamma(x+iy_r(x))|) = 0$ because $|\Gamma(x+iy_r(x))|$ is constant (equal to $r$).
    Recall that for a differentiable product of differentiable functions $f(x) = \prod_{k=0}^{\infty}f_{k}(x)$ we have $f'(x) = f(x)\sum_{n=0}^{\infty}\frac{f_{n}'(x)}{f_{n}(x)}$.
    Thus, we can write
    \begin{align*}
        0
            &= 
            \frac{d}{dx}\left(\frac{e^{-\gamma x}}{\sqrt{x^2+y_r^2}} \prod_{n=1}^{\infty} \frac{e^{x/n}}{\sqrt{\left(1+\frac{x}{n}\right)^2 + \left(\frac{y_r}{n}\right)^2}}\right) \\
            &=
            r\left(-\gamma -\frac{y_ry_r'+x}{x^2+y_r^2} + \sum_{n=1}^{\infty} \left(\frac{1}{n}-\frac{y_ry_r'+n+x}{(n+x)^2+y_r^2}\right)\right).
    \end{align*}
    Solving for $y_r'$ gives
    \[
        y'_r = \frac{-\gamma - \frac{x}{x^2+y_r^2} + \sum_{n=1}^{\infty}\left(\frac{1}{n}-\frac{n+x}{(n+x)^2+y_r^2}\right)}{\frac{y_r}{x^2+y_r^2}+\sum_{n=1}^{\infty}\frac{y_r}{(n+x)^2+y_r^2}}.
    \]
    We remark that since $x\geq 1$, then $\frac{x}{x^2+y_r^2}\leq \frac{1}{x}\leq 1$, and as $y_r\geq 0$, then $\frac{y_r}{x^2+y_r^2}\leq \frac{y_r}{1+y_r^2}\leq \frac{1}{2}$.
    Also, the sum $\sum_{n=1}^{\infty}\frac{y_r}{(n+x)^2+y_r^2}$ is bounded above by $y_r\frac{\pi^2}{6} = y_r\sum_{n=1}^{\infty} \frac{1}{n^2}$. So
    \begin{equation*}
        y'_r\geq \frac{-\gamma - 1 + \sum_{n=1}^{\infty}\left(\frac{1}{n}-\frac{n+x}{(n+x)^2+y_r^2}\right)}{\frac{1}{2}+y_r\frac{\pi^2}{6}}\geq -2\gamma - 2 + 2\sum_{n=1}^{\infty}\left(\frac{1}{n}-\frac{n+x}{(n+x)^2+y_r^2}\right).
    \end{equation*}
    The sum $\sum_{n=1}^{\infty}\left(\frac{1}{n}-\frac{n+x}{(n+x)^2+y_r^2}\right)$ gets arbitrarily large as $x \to +\infty$: 
    \begin{align*}
        \sum_{n=1}^{\infty}\left(\frac{1}{n}-\frac{n+x}{(n+x)^2+y_r^2}\right) 
            \ge
            \sum_{n=1}^{\infty}\left(\frac{1}{n}-\frac{1}{n+x}\right) 
            \ge
            \sum_{n=1}^{\infty}\left(\frac{1}{n}-\frac{1}{n+\lfloor x \rfloor}\right) 
            &=
            \sum_{n=1}^{\lfloor x \rfloor} \frac{1}{n} \\
            &\geq \log\left(\lfloor x\rfloor\right) + \gamma.
    \end{align*}
    Altogether,
    \begin{align*}
        y_r' 
            &\ge -2\gamma - 2 + 2\sum_{n=1}^{\infty}\left(\frac{1}{n}-\frac{n+x}{(n+x)^2+y_r^2}\right) \\
            &\ge -2\gamma - 2 + 2\left(\log\left(\lfloor x\rfloor\right) + \gamma\right) \\
            &=2\log\left(\lfloor x\rfloor\right) - 2.\qedhere
    \end{align*}
\end{proof}

Next, we show that along each curve $C_r$ in $Q_{\alpha,0}$, $\Gamma(x+iy(x))$ cycles around the circle of modulus $r$ faster as $x$ gets larger.

\begin{lemma}\label{ArgRateOfChange}
    Let $r \in \RR^+$ and let $y_r(x)$ be the function given by Proposition \ref{ConstModulusCurves}. 
Then 
\[\frac{d}{dx}(\arg\Gamma(x+iy_r)) \geq 2(\log(\lfloor x \rfloor)-1)^2. \]
    In particular, $\displaystyle \lim_{x \to +\infty} \frac{d}{dx}\left(\arg \Gamma (x+iy_r(x))\right) = +\infty$.
\end{lemma}

\begin{proof}
    Since $\log \Gamma(z) = \log |\Gamma(z)| + i \arg \Gamma(z)$, we get the following expression for $\arg \Gamma(z)$ along $C_r$:
    \[
        \arg \Gamma(x+iy_r) = -\gamma y_r - \arctan\left(\frac{y_r}{x}\right) + \sum_{n=1}^{\infty} \left(\frac{y_r}{n} - \arctan\left(\frac{y_r}{n+x}\right)\right).
    \]
    Differentiating with respect to $x$ and writing $y_r'$ to denote $y_r'(x)$, we get 
    \begin{align*}
        \frac{d}{dx}(\arg\Gamma(x+iy_r)) 
            =& 
            -\gamma y_r' - \frac{xy_r'-y_r}{x^2+y_r^2} + \sum_{n=1}^{\infty} \left(\frac{y_r'}{n} - \frac{(n+x)y_r'-y_r}{(n+x)^2+y_r^2}\right)\\
            =&
            y_r'\left(-\gamma -\frac{x}{x^2+y_r^2} + \sum_{n=1}^{\infty}\left(\frac{1}{n}-\frac{n+x}{(n+x)^2+y_r^2}\right)\right) \\
            &\qquad + \frac{y_r}{x^2+y_r^2} + \sum_{n=1}^{\infty}\frac{y_r}{(n+x)^2+y_r^2} \\
            \ge& (2\log (\lfloor x \rfloor)-2)\left(-\gamma -1 + \log(\lfloor x \rfloor)+\gamma\right) \\
            =& 2(\log(\lfloor x \rfloor)-1)^2. \qedhere
    \end{align*}
\end{proof}

\begin{notn}
\label{notation:ast}
For each $z \in Q_{\alpha,0}$, define $z^\ast$ to be the the number of least modulus greater than $|z|$ in $Q_{\alpha,0}$ such that $\Gamma(z^\ast) = \Gamma(z)$.    
\end{notn}

\begin{cor}\label{GammaArgCycle}
    For all $z \in Q_{\alpha,0}$, we have $\displaystyle |z - z^\ast|<\frac{\pi}{(\log(\lfloor \Re(z) \rfloor)-1)^2}$. 
\end{cor}

\begin{proof}
    Let $r = |\Gamma(z)|$ and let $y_r$ be the function whose graph is $C_r \subset Q_{\alpha,0}$.
    Then $z^{\ast}$ is the first point in $C_r$ with $\Re(z^{\ast})>\Re(z)$ and $\arg \Gamma(z^{\ast}) = \arg \Gamma(z)$.
    Note that $|z-z^{\ast}|$ is less than the length of the segment of $C_r$ between $z$ and $z^{\ast}$.
    By Lemma~\ref{ArgRateOfChange}, $\displaystyle \frac{d}{dx}\left(\arg \Gamma (x+iy_r(x))\right) > 2(\log(\lfloor x \rfloor)-1)^2$.
    So by the mean value theorem, 
    \[
        |z-z^{\ast}| \le \frac{2\pi}{2(\log(\lfloor x \rfloor)-1)^2} = \frac{\pi}{(\log(\lfloor x \rfloor)-1)^2} \qedhere
    \]
\end{proof}

\begin{rmk}
    We will use the bound found in Corollary~\ref{GammaArgCycle} for large $\Re(z)$.
    The bound can be improved for small values of $|z|$.
    For example, if $\Re(z),\Im(z) \ge 3$, then it can be shown that $|z - z^{\ast}|<2\pi$.
\end{rmk}

Next we study the family of sets $\{z\in Q_{\alpha,0} : \arg\Gamma(z) = \theta\}$ for fixed $\theta \in (-\pi,\pi]$.

\begin{prop}
\label{prop:fixarg}
    For each $\theta \in (-\pi,\pi]$, the set of $x+iy \in Q_{\alpha,0}$ such that $\arg\Gamma(x+iy) = \theta$ is a  collection of disjoint $C^1$ curves, each of which is the graph of a function $y_\theta(x)$ whose slope is negative and approaches zero as $x \to +\infty$.
\end{prop}

\begin{proof}
    Let $\theta \in (-\pi,\pi]$.
    Let $z \in Q_{\alpha,0}$ and let $r = |\Gamma(z)|$.
    By Lemma~\ref{ArgRateOfChange}, there is a unique point $z_0$ in the segment of $C_r$ between $z$ and $z^{\ast}$ such that $\arg \Gamma(z_0) = \theta$.
    Let $x_0 = \Re(z_0)$ and $y_0 = \Im(z_0)$. 
    Since curves of fixed argument and modulus in $\CC$ intersect at right angles, and since $\Gamma$ is locally invertible on $Q_{\alpha,0}$ (see \cite{uchiyama}) and conformal, there must be a neighborhood $U \ni x_0+iy_0$ such that $U \cap \{x+iy :\arg\Gamma(x+iy) = \theta\}$ is a curve $K_{\theta}$ that intersects $C_r$ at $x_0+iy_0$ at a right angle. 
    Since the slope of $C_r$ at $x_0+iy_0$ is positive (by Lemma~\ref{Limity'toInfty}), $K_{\theta}$ must pass through $x_0+iy_0$ with negative slope.
    Specifically, if $y_r(x)$ is the function whose graph is $C_r$, then the slope of $K_{\theta}$ at $x_0+iy_0$ is $\frac{-1}{y_r'(x)}$.
    Thus $\frac{\partial \arg\Gamma(x+iy)}{\partial y}(x_0+iy_0) \ne 0$, and we can apply the implicit function theorem to obtain $K_{\theta}$ as the graph of a unique $C^1$ function $y_\theta(x)$ with $y_\theta'(x) =\frac{-1}{y_r'(x)} <0$.
    
    We first show that we can extend $y_\theta(x)$ on the left to the line $x = \alpha$.
    Applying the above argument at every $x+iy$ in $Q_{\alpha,0}$ shows that $\frac{\partial \arg\Gamma(x+iy)}{\partial y}$ is never 0.
    Suppose $y_\theta(x)$ has a vertical asymptote at $x=a <x_0$.
    By conformality, that would mean the slopes of $C_r$ curves in a neighborhood of $x=a$ would get arbitrarily small as $y$ gets arbitrarily large.
    However, by Lemma~\ref{Limity'toInfty}, $y_r'(x)\geq 2\log(\lfloor x\rfloor)-2$, contradicting that $y'_r$ can get arbitrarily small.

    The only barrier to extending $y_\theta(x)$ on the right to the whole interval $(x_0,+\infty)$ would be if $y_\theta(x)$ intersects the real axis for some value of $x$.
    But this cannot happen because the positive reals form a curve of fixed argument for $\Gamma$ (since $\Gamma$ is real valued and positive on $(x_0,+\infty)$) and different fixed argument curves cannot intersect, since neither $\frac{\partial \arg\Gamma(x+iy)}{\partial y}$ nor $\Gamma(x+iy)$ vanish in $Q_{\alpha,0} \cup \{x\in\RR : x >\alpha\}$. 
\end{proof}

\begin{cor}
    \label{cor:modonfixarg}
    Given $\theta \in (-\pi,\pi]$, let $y_\theta(x)$ be a function such that for all $x>\alpha$ we have $\arg(\Gamma(x+iy_\theta(x))) = \theta$. Then $|\Gamma(x+iy_\theta(x))|$ grows at least exponentially to $+\infty$ as $x$ does.
\end{cor}
\begin{proof}
    By (\ref{eq:stirlingineq}) we know that
    \begin{align*}
        |\Gamma(x+iy_\theta(x))| &\geq \frac{\sqrt{2\pi}}{e}|z|^{x-\frac{1}{2}}\exp(-y_\theta(x)\arg(x+iy_\theta(x))-x)\\
        &= \frac{\sqrt{2\pi}}{e}\exp\left(x(\log|z|-1) - \frac{1}{2}\log|z|-y_\theta(x)\arg(x+iy_\theta(x))\right).
    \end{align*}
    By Proposition \ref{prop:fixarg} we know that the function $y_\theta(x)$ is both decreasing and positive, which implies that $y_\theta(x)$ is bounded as $x\to+\infty$. On the other hand $\arg(x+iy_\theta(x))$ is bounded between $\left(0,\frac{\pi}{2}\right)$ since $x+iy_\theta(x)\in Q_{\alpha,0}$, therefore $y_\theta(x)\arg(x+iy_\theta(x))$ is bounded as $x\to+\infty$, which completes the proof. 
\end{proof}

\subsection{Transcendence conjectures for \texorpdfstring{$\Gamma$}{Gamma}}
For many of the periodic functions we mentioned in the introduction there are precise conjectures on the geometric conditions needed to ensure that an algebraic variety $V$ has a Zariski dense intersection with the graph of a given function $f$. 
In the case of exponential maps the conditions are usually referred to as \emph{rotundity} and \emph{freeness} (see \cite[Definition~7.1]{bays-kirby}), whereas in the case of automorphic functions the conditions are called \emph{broadness} and \emph{freeness} (see \cite[Definition~2.3]{aslanyan-kirby} for the definitions in the case of the $j$ function, and see \cite[\S1.1]{eterovic-zhao} for the general case). 
The legitimacy of these conditions rests upon well-known conjectures (such as Schanuel's conjecture for exponentiation) and important theorems in functional transcendence (usually referred to as Ax--Schanuel theorems). 
No such functional transcendence statement is known for $\Gamma$, but there are some well-known transcendence conjectures about its values, such as the following.  
\begin{conj}[Lang--Rohrlich, see {{\cite[Conjecture~3.3]{gun-murty-rath}}}]
\label{conj:lang-rohrlich}
    Let $q>2$ be an integer. Then 
    \begin{equation*}
        \mathrm{tr.deg.}_{\QQ}\QQ\left(\{\pi\}\cup \left\{\Gamma\left(\frac{a}{q}\right) : a\in\{1,\ldots,q-1\}\wedge \mathrm{gcd}(a,q)=1\right\}\right)\geq 1+\frac{\phi(q)}{2},
    \end{equation*}
    where $\phi$ denotes Euler's totient function.
\end{conj}
However, this conjecture does not lend itself to a formulation of geometric conditions in the same way that Schanuel's conjecture 
does for complex exponentiation. 
So during the remainder of this section we will present some statements based on suggestions that have been communicated to us by Jonathan Kirby that address this. 
The first is a transcendence statement for $\Gamma$ in the spirit of Schanuel's conjecture, Conjecture~\ref{conj:schanuelgamma}, and the second is a functional transcendence conjecture akin to Ax--Schanuel, Conjecture~\ref{conj:weakas}. 
They can be seen as a supplement to the Lang--Rohrlich conjecture as they will avoid considering rational values for the argument of $\Gamma$, and they serve as a guideline for the eventual conjecture on the geometric conditions needed for $\Gamma$ (Conjecture~\ref{conj:ecgamma}). 

We first remark that Schanuel's conjecture can be phrased as a comparison between two notions of dimension: one given by algebraic relations (transcendence degree), and the other given by $\QQ$-linear dependencies. 
Indeed, given a finite subset $A\subset\mathbb{C}$ we write $\mathrm{l.dim}_{\QQ}(A)$ to denote the linear dimension of the $\QQ$-linear span generated by the elements of $A$, as a $\QQ$-vector subspace of $\CC$. 
Then Schanuel's conjecture can be stated as follows.
\begin{conj}[Schanuel, see {{\cite[p. 30--31]{lang2}}}]
    For every $z_1,\ldots,z_n\in\CC$ we have
    \begin{equation*}
        \mathrm{tr.deg.}_{\QQ} \QQ(\mathbf{z},\exp(\mathbf{z}))\geq \mathrm{l.dim}_{\QQ}(\mathbf{z}).
    \end{equation*}
\end{conj}
The appearance of the $\QQ$-linear dimension accounts for the well-known fact that $\QQ$-linear dependencies among complex numbers are transformed into multiplicative dependencies under exponentiation. 

In the case of $\Gamma$, there are several examples of multiplicative dependencies among values of $\Gamma$ at rational numbers and certain algebraic numbers. 
For example, 
\begin{equation}
\label{eq:gammaalgebraic}
    \Gamma\left(\frac{1}{5}\right)\Gamma\left(\frac{4}{15}\right)\sqrt[6]{5}\sqrt[4]{5-\frac{7}{\sqrt{5}}+\sqrt{6-\frac{6}{\sqrt{5}}}} = \Gamma\left(\frac{1}{3}\right)\Gamma\left(\frac{2}{15}\right)\sqrt{2}\sqrt[20]{3}.
\end{equation}
Identities such as (\ref{eq:gammaalgebraic}) are consequences of the three difference equations mentioned earlier, especially (\ref{eq:diffeq3}). 
The hypotheses in the following conjecture have been included to prevent the known difference equations from occurring.
 \begin{conj}
\label{conj:schanuelgamma}
    Let $z_1,\ldots,z_n\in\CC$ be such that for every $i,j\in\{1,\ldots,n\}$, with $i\neq j$, we have $z_i\notin \QQ z_j + \QQ$. Then 
    \begin{equation*}
        \mathrm{tr.deg.}_{\QQ}\QQ(\mathbf{z},\Gamma(\mathbf{z}))\geq n.
    \end{equation*}
\end{conj}

\subsection{Geometric properties of \texorpdfstring{$\Gamma$}{Gamma}}

We warn the reader that the terminology defined in this section will be seldom used later. 
Instead, the conjectures presented here are intended to provide context and motivation for the main result.

We recall that the general problem one would like to solve is to find a minimal set of geometric conditions ensuring that an algebraic variety $V$ has a Zariski dense intersection with the graph of some transcendental function $f$. 
To find the geometric conditions for a given function, in our case $\Gamma$, one needs to understand the functional properties of the function. 
For example, the complex exponential function satisfies $\exp(a+b)= \exp(a)\exp(b)$ for all $a,b\in\CC$. 
This property then has the following consequence: if $L$ is any $\mathbb{Q}$-linear subspace of $\CC^n$ and $\mathbf{a}$ is any point in $\CC^n$, then $L+\mathbf{a}$ is an algebraic subvariety of $\CC^n$ with the property that $\exp(L+\mathbf{a})$ is also an algebraic variety, specifically, $\exp(L+\mathbf{a})$  is a coset of a subtorus of $\left(\CC^\times\right)^n$. 
Borrowing terminology from the world of unlikely intersections, these types of varieties are the \emph{weakly special} varieties of the exponential function. 
In the case of $\Gamma$, the difference equations inform what some of those varieties are. 
We formulate this more precisely with the following definition.

\begin{defn}
    Let $W\subseteq\CC^n$ be an algebraic subvariety. We say that $W$ is $\Gamma$-\emph{weakly special} if
    \begin{itemize}
        \item $W = \CC^n$, or
        \item $W$ is a proper subvariety of $\CC^n$ and the set
    \begin{equation}
    \label{eq:gamaspecial}
        \{(\mathbf{z},\Gamma(\mathbf{z})) : \mathbf{z}\in W\}
    \end{equation}
    is contained in a proper subvariety of $W\times\CC^{n}$. In this case we write $\Gamma_W$ to denote the Zariski closure of (\ref{eq:gamaspecial}) in $\CC^{2n}$. 
    \end{itemize}
    
    We say that $W$ is $\Gamma$-\emph{special} if
    \begin{itemize}
        \item it is weakly special, and
        \item if $W$ is a proper subvariety of $\CC^n$, then both $W$ and $\Gamma_W$ are defined over $\overline{\mathbb{Q}}$. 
    \end{itemize}
    In particular $\CC^n$ is $\Gamma$-special. 
\end{defn}
We remark that unlike the weakly special subvarieties associated with exponential or automorphic functions, the $\Gamma$-weakly special varieties do not have the \emph{bi-algebraicity property}, that is, that $W$ is $\Gamma$-weakly special \textbf{does not} imply that $\Gamma(W)$ is also an algebraic variety. See Example \ref{ex:gammaspecial}.

When we work with subvarieties of $\CC^{2n}$ with the purpose of looking at their possible intersection with the graph of $\Gamma$, we will always think of these varieties as defined by polynomials in the ring $\CC[X_1,\ldots,X_n,Y_1,\ldots,Y_n]$, where the $X_i$ represent the first $n$ coordinates of $\CC^{2n}$, and the $Y_i$ represent the last $n$ coordinates. 

We start by considering some trivial examples of $\Gamma$-weakly special varieties.

\begin{ex}
    Let $W$ be the diagonal of $\CC^{n}$, that is, $W$ is defined by
    \begin{equation*}
        X_1 = X_2=\cdots =X_n.
    \end{equation*}
    Then $W$ is $\Gamma$-special. Indeed, in this case $\Gamma_W$ is simply the subvariety of $\CC^{2n}$ defined by
    \begin{equation*}
        X_1 = X_2=\cdots =X_n \quad\mbox{ and }\quad Y_1=Y_2=\cdots=Y_n.
    \end{equation*}
    We remark this example is suitable for any (set-theoretic) function.
\end{ex}

\begin{ex}
\label{ex:weaksppts}
    Any point $\mathbf{z}\in\CC^n$ is trivially $\Gamma$-weakly special because its image $\Gamma(\mathbf{z})$ is a 0-dimensional algebraic variety. 
    Like in the previous example, this is just a property of functions.

    It would be really interesting to understand the $\Gamma$-special points. 
    Since $\Gamma(1)=1$ and we have (\ref{eq:diffeq1}), we know that every positive integer $n$ is $\Gamma$-special because $\Gamma(n)=(n-1)!$. 
    It seems to be unknown if there are any other $\Gamma$-special points.\footnote{For example, it seems to be unknown whether $\Gamma\left(\frac{1}{5}\right)$ is transcendental over $\mathbb{Q}$, although this is predicted by Conjecture~\ref{conj:lang-rohrlich}.}
\end{ex}

Now we present an example that is more particular to $\Gamma$.
\begin{ex}
\label{ex:gammaspecial}
    The variety $W\subset\CC^{2}$ defined by the equation $X_2 = X_1+1$ is $\Gamma$-special since $\Gamma_W$ is then contained in the variety defined by
    \begin{equation}
    \label{eq:specialdiffeq}
        X_2 = X_1+1\quad\mbox{ and }\quad Y_2 = X_1Y_1.
    \end{equation}
    The equation (\ref{eq:diffeq1}) shows that $\Gamma_W$ is contained in (\ref{eq:specialdiffeq}), and hence $W$ is at least $\Gamma$-weakly special. 
    But since $\Gamma$ is a transcendental function, the Zariski closure of the graph $\left\{(z,\Gamma(z)) : z\in\CC\setminus\mathbb{Z}^{\leq 0}\right\}$ has dimension 2.
    Therefore, $\Gamma_W$ must have dimension at least 2. 
    Since (\ref{eq:specialdiffeq}) defines a variety of dimension 2 and is defined over $\mathbb{Q}$, this shows that $W$ is $\Gamma$-special and $\Gamma_W$ is given by (\ref{eq:specialdiffeq}).
    
    We point out, however, that the image of $W$ under $\Gamma$ is not an algebraic variety. 
    Indeed, $\Gamma(W) = \{(\Gamma(z),z\Gamma(z)) : z\in\CC\}$, which shows that $\Gamma(W)$ has complex analytic dimension 1. 
    Since the functions $\Gamma(z)$ and $z$ are algebraically independent over $\CC$, $\Gamma(W)$ cannot be contained in an algebraic curve.
\end{ex}

The difference equations (\ref{eq:diffeq2}) and (\ref{eq:diffeq3}) involve the sine and exponential functions, and as such they do not describe algebraic geometric properties of $\Gamma$, or in other words, these symmetries of $\Gamma$ cannot be witnessed only using algebraic varieties the way we did with (\ref{eq:diffeq1}) in the last example.

It should be noted that unlike the difference equations satisfied by periodic functions, such as the exponential maps of semi-abelian varieties or the uniformization maps of Shimura varieties, the equation (\ref{eq:diffeq1}) involves both the argument and the values of the function. 
Because of this, our definition of $\Gamma$-weakly special is about subvarieties of the domain of $\Gamma$, whereas for the other functions we have mentioned, the definition of weakly special varieties is about subvarieties of the codomain of the function. 

The reason behind our interest in $\Gamma$-weakly special varieties is that their presence may prevent algebraic varieties from ever intersecting the graph of $\Gamma$, as the following examples illustrate.

\begin{ex}
    Define $V\subseteq\CC^4$ as 
    \begin{equation*}
        X_2 = X_1\quad\mbox{ and }\quad Y_2 = Y_1+1.
    \end{equation*}
    then $V$ cannot have points in the graph of any function.
\end{ex}

\begin{ex}
Define $V\subseteq\CC^4$ as 
    \begin{equation*}
        X_2 = X_1+1\quad\mbox{ and }\quad Y_2 = X_1Y_1+1.
    \end{equation*}
    then $V$ cannot have points in the graph of $\Gamma$ as that would contradict (\ref{eq:diffeq1}).
\end{ex}

In view of the Lang--Rohrlich conjecture or Conjecture~\ref{conj:schanuelgamma}, it is not expected that one can find any other $\Gamma$-weakly special varieties. 
Both Lang--Rohrlich and Schanuel's conjecture remain open, but in the case of $\exp$ we do have a complete understanding of its weakly special varieties. 
This is due to Ax's theorem \cite[Theorem 3]{ax} (also known as the Ax--Schanuel theorem), although only a weaker statement is needed (sometimes called the Ax--Lindemann--Weierstrass theorem). 
Similarly, we understand the weakly special varieties of any function for which an Ax--Schanuel theorem is known. 
We then make the following conjecture, which (being functional) could be more tractable than the other conjectures we have mentioned.
\begin{conj}[weak Ax--Schanuel for $\Gamma$]
\label{conj:weakas}
    Let $f_1,\ldots,f_n$ be holomorphic functions such that for every $i,j\in\{1,\ldots,n\}$, with $i\neq j$, we have $f_i\notin \QQ f_j + \CC$. Then 
    \begin{equation*}
        \mathrm{tr.deg.}_{\CC}\CC(\mathbf{f},\Gamma(\mathbf{f}))\geq n + 1.
    \end{equation*}
\end{conj}

\begin{defn}
    A subvariety $V\subset\CC^{2n}$ is said to be $\Gamma$-\emph{free} if it is not contained in any proper subvarieties of the form $\Gamma_W$, where $W$ is $\Gamma$-weakly special.
\end{defn}
As we mentioned in Example~\ref{ex:weaksppts}, every point is $\Gamma$-weakly special, so the definition of a $\Gamma$-free variety in particular implies that $V$ has no constant coordinates.

The weak Ax--Schanuel conjecture for $\Gamma$ tells us that if $V$ has no constant coordinates and it is not contained in any subvariety of $\CC^{2n}$ of the form $X_i + a_1 X_j + a_2=0$, where $a_1,a_2\in\QQ$, then $V$ is $\Gamma$-free. 

Let $n,\ell$ be positive integers with $\ell \leq n$ and $\mathbf{i}=(i_1,\ldots,i_\ell)$ in $\mathbb{N}^{\ell}$ with $1\leq i_1 < \ldots < i_\ell \leq n$. 
Define the projection map $\mathrm{pr}_{\mathbf{i}}:\CC^{n} \rightarrow \CC^{\ell}$ by $\mathrm{pr}_{\mathbf{i}}(x_1,\ldots,x_n):= (x_{i_1},\ldots,x_{i_\ell})$.

\begin{rmk}
It is easy to see that in the examples we have given above where $W$ is a $\Gamma$-weakly special subvariety of $\CC^{n}$, for any choice of indices $1\leq i_{1}<\cdots<i_{\ell}\leq n$ we have that $\mathrm{pr}_{\mathbf{i}}(W)$ is also a $\Gamma$-weakly special subvariety of $\CC^{\ell}$. 
\end{rmk}

Define $\mathrm{Pr}_{\mathbf{i}}:\CC^{2n}\to \CC^{2\ell}$ by $\mathrm{Pr}_{\mathbf{i}}(\mathbf{x},\mathbf{y}):= (\mathrm{pr}_{\mathbf{i}}(\mathbf{x}),\mathrm{pr}_{\mathbf{i}}(\mathbf{y}))$.
The following definition is taken from \cite{aek-differentialEC}.
\begin{defn}
An algebraic set $V \subseteq \CC^{2n}$ is said to be \emph{broad} if for any $\mathbf{i}=(i_1,\ldots,i_{\ell})$ in $\mathbb{N}^{\ell}$ with $1\leq i_1 < \cdots < i_{\ell} \leq n$ we have $\dim \mathrm{Pr}_{\mathbf{i}} (V) \geq \ell$. 
\end{defn}

In particular, if $V$ is broad then $\dim V\geq n$. We now formulate a possible answer to the problem of finding geometric conditions ensuring that $V$ has a Zariski dense intersection with the graph of $\Gamma$. 

\begin{conj}
\label{conj:ecgamma}
    Let $n$ be a positive integer and let $V\subseteq\CC^{2n}$ be an irreducible subvariety which is broad and is not contained in any subvariety of the form $X_i + a_1 X_j + a_2=0$, where $a_1,a_2\in\QQ$. 
    Then the intersection of $V$ with the graph of $\Gamma$ is Zariski dense in $V$.
\end{conj}

The conditions in Theorem~\ref{thm:ecforgamma} require that $\dim \pi_1(V)=n$ and that $V$ have no constant coordinates. 
This implies that $V$ is broad and that it is not contained in a $\Gamma$-special variety of the form presented in Example~\ref{ex:gammaspecial}. 
Thus, Theorem~\ref{thm:ecforgamma} provides evidence for Conjecture~\ref{conj:ecgamma}.

\section{The case of Plane Curves}
\label{sec:n=1}

In this section we will focus on the case $n=1$ of Theorem~\ref{thm:ecforgamma}. The key part of our method for the proof of Theorem \ref{thm:ecforgamma} is already contained in this case.

\subsection{A general result about differentially transcendental functions}
\label{subsec:difftrans}
\begin{defn}
    A holomorphic function $f$ is \emph{differentially transcendental} over $\CC(z)$ if for every non-negative integer $n$ and every non-constant polynomial $p\in\CC[X,Y_0,Y_1,\ldots,Y_n]$, the function $p\left(z,f(z),f'(z),\ldots,f^{(n)}(z)\right)$ is not identically zero.
\end{defn}

\begin{defn}
    An entire function $f(z)$ is said to have \emph{finite order of growth} if there are $A,B,\rho\in\mathbb{R}^+$ such that for all $z\in\mathbb{C}$ we have
    \begin{equation*}
        |f(z)|\leq A\exp\left(B|z|^\rho\right). 
    \end{equation*}
    We extend this definition and say that a meromorphic function $f$ on $\CC$ with finitely many poles has finite order of growth if there is a non-zero polynomial $h(z)\in\CC[z]$ such that $h(z)f(z)$ is entire and has finite order of growth.  
\end{defn} 

\begin{theorem}
\label{thm:difftrans}
    Let $f(z)$ be a meromorphic function on $\CC$ with finitely many poles and let $p(X,Y)\in\CC[X,Y]$ be an irreducible polynomial which depends at least on $Y$. 
    If $f(z)$ is differentially transcendental over $\CC(z)$ and has finite order of growth, then the set $\{z\in\CC: p(z,f(z))=0\}$ is infinite. 
\end{theorem}
\begin{proof}
    We proceed by contradiction, so suppose that $\{z\in\CC: p(z,f(z))=0\}$ is finite. 
    Observe that the function $p(z,f(z))$ has finitely many poles, so there exists a non-zero polynomial $h(z)\in\CC[z]$ such that $h(z)p(z,f(z))$ is entire, and since $f$ has finite order of growth, then so does $h(z)p(z,f(z))$.
    Then by Hadamard's factorization theorem (see e.g.~\cite[Chapter~5, Theorem~5.1]{complexanalysis}) we can find a polynomial $q(z)\in\CC[z]$ and a rational function $R(z)\in\CC(z)$ such that 
    \begin{equation*}
        p(z,f(z)) = R(z)\exp(q(z)).
    \end{equation*}
    However, by hypothesis, the left hand side of this equality is differentially transcendental over $\CC(z)$, whereas the right hand side is differentially algebraic over $\CC(z)$. 
\end{proof}

In particular this result immediately applies to the Riemann $\zeta$ function (and other Dirichlet series), and as such it complements \cite[Theorem~1.2]{zetaEC}.
 We will use it to give a proof for $\Gamma$ in the next subsection.

We remark that the proof of Theorem~\ref{thm:difftrans} can be adapted for other functions of finite order of growth, even some that are not differentially transcendental. 
For example, it can be used to obtain an analogous result for the exponential function, as shown in \cite[Corollary~2.4]{marker}

\subsection{The result for \texorpdfstring{$\Gamma$}{Gamma}}

We will now prove the case $n=1$ of Theorem~\ref{thm:ecforgamma}. We give two proofs of this result. The first one uses Theorem~\ref{thm:difftrans}, and on its own does not provide information about the distribution of the solutions. 

The second proof is an explicit method which will allow us to obtain a lower bound for the number of solutions inside balls of a given radius (see Corollary~\ref{DistributionLowerBd}). This proof will illustrate the constructions and methods used in the proof of Theorem~\ref{thm:ecforgamma} in the next section, while avoiding some of the technical complications that arise in higher dimensions which can obscure the general strategy.

\begin{theorem}\label{OneVarEC}
    Let $p(X,Y)\in\CC[X,Y]$ be an irreducible polynomial depending at least on $Y$, which is not in $\CC Y$. Then the set $\{z\in \CC : p(z,\Gamma(z))=0\}$ is infinite.
\end{theorem}

\begin{proof}[First proof of Theorem~\ref{OneVarEC}]
    It is well-known that $\Gamma$ is differentially transcendental (see \cite[Chapter~2, \S 2.6]{remmert}). While $\Gamma$ has infinitely many poles (at all non-positive integers), $\frac{1}{\Gamma}$ is an entire function with order of growth 1 (\cite[Chapter~6, Theorem~1.6]{complexanalysis}). So by Theorem~\ref{thm:difftrans} we get that the set $\left\{z\in\CC : h\left(z,\frac{1}{\Gamma(z)}\right)\right\}$ is infinite for every irreducible polynomial $h(X,Y)\in\CC[X,Y]$. By clearing denominators for $\Gamma$, we obtain the result (here we require that the polynomial not be of the form $aY$ for some $a\in\CC)$. 
\end{proof}

Before giving the second proof, we introduce the following definitions.

\begin{defn}
    Given a point $\beta\in Q_{\alpha,0}$ and given $R>|\Gamma(\beta)|$, let $\rho(\beta,R)$ denote the point in $\{z\in Q_{\alpha,0}:\arg(\Gamma(z)) = \arg(\Gamma(\beta))\}$ which lies in the same component as $\beta$, and satisfies $|\Gamma(\rho(\beta,R))|=R$. 

    We also define $T(\beta,R)$ to be the bounded segment between $\beta$ and $\rho(\beta,R)$ along the component of $\{z\in Q_{\alpha,0}:\arg(\Gamma(z)) = \arg(\Gamma(\beta))\}$ which contains $\beta$.
    
    Given two points $\beta_1, \beta_2\in Q_{\alpha,0}$ satisfying $|\Gamma(\beta_1)| = |\Gamma(\beta_2)|$, define $S(\beta_1,\beta_2)$ to be the bounded segment between $\beta_1$ and $\beta_2$ within the curve $C_{|\Gamma(\beta_1)|}$.
    
    We define $K(\beta,R)$ as the simple closed curve defined by the concatenation of the following four segments: $S(\beta,\beta^\ast)$, $T(\beta,R)$, $T(\beta^\ast,R)$, $S(\rho(\beta,R),\rho(\beta^\ast,R))$.\footnote{We recall that the notation $\beta^\ast$ was introduced in Notation \ref{notation:ast}.}
\end{defn}

\begin{figure}[h]
\centering
\includegraphics[width=0.5\textwidth]{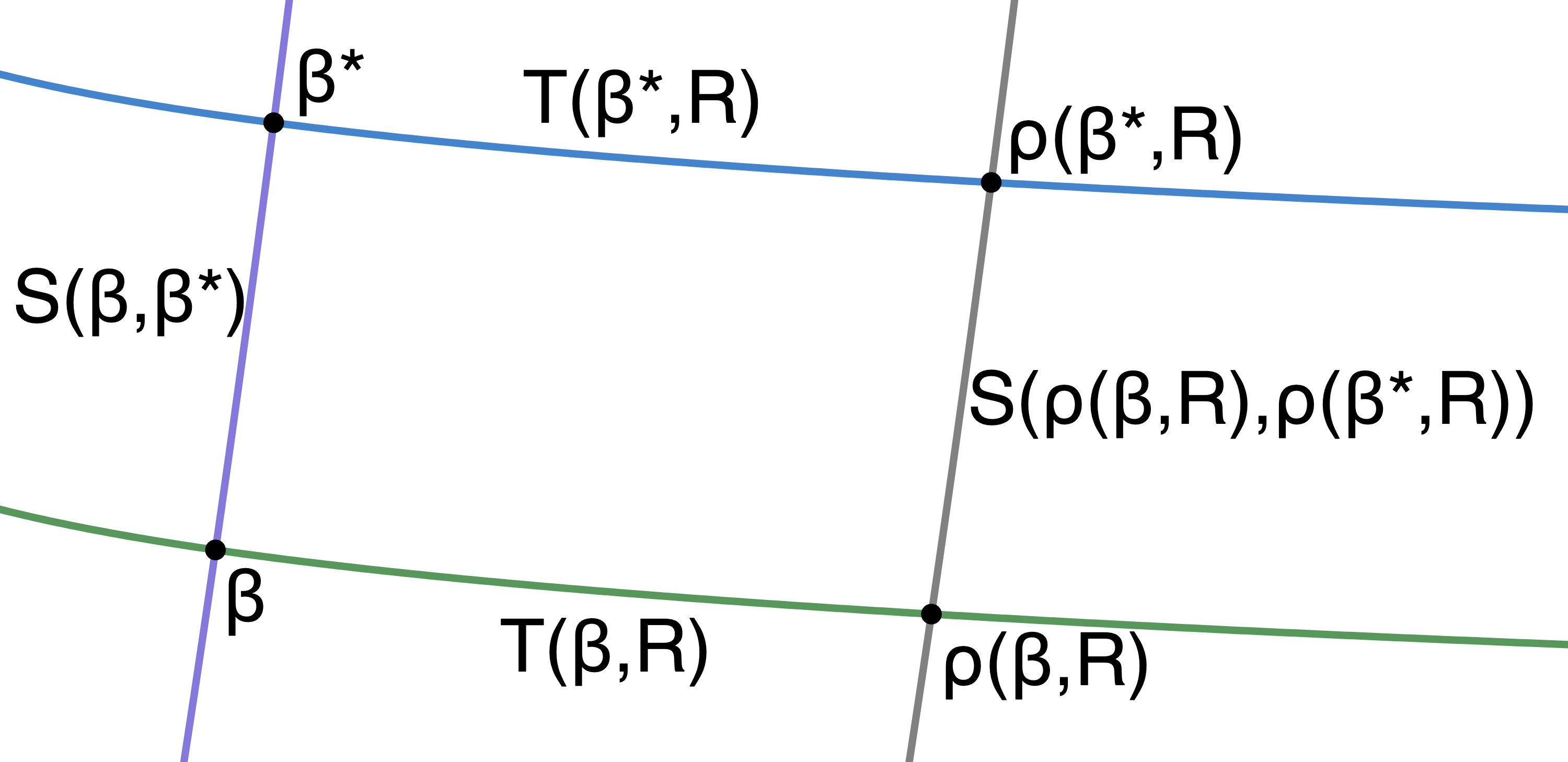}
\caption{Diagram of the curve $K(\beta,R)$.}
\label{fig:Kcurve}
\end{figure}

\begin{proof}[Second proof of Theorem~\ref{OneVarEC}]
    It suffices to prove that the function $\Gamma(z)-A(z)$ has infinitely many zeros, where $A(z)$ is a non-zero algebraic function. 
    We may assume that $A$ is holomorphic on $Q_{\alpha,0}\setminus F$, where $F$ is some compact set (for the basics of algebraic functions in one variable, see \cite[pp.~284--308]{ahlfors}).
    
    Let $D$ be a disc centered at the origin which contains all zeroes and poles of $A(z)$. 
    To prove the theorem, we will use the Argument Principle. 
    Specifically, we will construct simple closed curves $K$ in $Q_{\alpha,0}\setminus D$, so that the winding number of $(\Gamma-A)(K)$ around the origin is 1. 
    Since we are avoiding $D$ and $\Gamma$ has no poles in $Q_{\alpha,0}$, showing that the winding number is nonzero is equivalent in this case to proving that $\Gamma-A$ has a zero in the region bounded by $K$.
     
    Observe now that for $w\in\CC$ given, $\lim_{z\to\infty}\left|\frac{A(z+w)}{A(z)}-1\right|=0$, where $z$ tends to infinity within $Q_{\alpha,0}$, and that the limit is uniform if $w$ is restricted to a compact set. 
    So there is $N$ such that for all $z \in Q_{\alpha,0}$ with $|z|\ge N$ and all $w \in \CC$ with $|w|<3$,
    \begin{equation}
    \label{eq:A(z)/4}
        |A(z+w)-A(z)| < \frac{|A(z)|}{4}.
    \end{equation}
    
    The set $G:=\{z\in Q_{\alpha,0} : z\notin D \wedge \Re(z)\geq\max\{16, N\}\}$ is simply connected. 
    The function $z\mapsto |\Gamma(z)| - \frac{|A(z)|}{4}$ is continuous on $G$.
    By Corollary~\ref{GammaLimitXtoInfty}, this function is positive when $\Im(z)$ is fixed and $\Re(z)$ is taken large enough, and by Corollary~\ref{GammaLimitYtoInfty} it is negative when $\Re(z)$ is fixed and $\Im(z)$ is large enough. 
    So there is $\xi\in G$ such that 
    \begin{equation}
    \label{eq:Gamma(beta_0)}
        |\Gamma(\xi)| = \frac{|A(\xi)|}{4}.
    \end{equation}
    In particular, $A(\xi)\neq 0$ as $\xi\notin D$.
    Let $\xi = x+iy$, where $x,y\in\RR$ and let $r := |\Gamma(\xi)|$.
    For each $z \in S(\xi,\xi^\ast)$, 
    \[
        |\Gamma(z)| = r = |\Gamma(\xi)| = \frac{|A(\xi)|}{4}.
    \]
    Since $x \ge 16$, Corollary \ref{GammaArgCycle} shows that for any $z \in S(\xi,\xi^{\ast})$, we have $|z-\xi| \le \frac{\pi}{(\log \lfloor 16\rfloor-1)^2} < 1$.
    Since also $x \ge N$, (\ref{eq:A(z)/4}) tells us that for any $z \in S(\xi,\xi^{\ast})$, we have $|A(z)-A(\xi)|<r$, i.e., $-A(S(\xi,\xi^\ast))\subset B(-A(\xi),r)$.
    Similarly, for every $z\in S(\xi,\xi^\ast)$, $|\Gamma(z)-A(z) - (-A(\xi))|\leq r + |A(z)-A(\xi)| < 2r$, so $(\Gamma-A)(S(\xi,\xi^\ast))\subset B(-A(\xi),2r)$.
    
    As we move along $S(\xi,\xi^\ast)$, the argument of $\Gamma(z)$ cycles once through $(-\pi,\pi]$, so we may pick $\beta$ in $S(\xi,\xi^\ast)$ such that 
    \[
        \arg\Gamma(\beta) = \arg(-A(\xi)).
    \]
    For any $z \in S(\beta,\beta^{\ast})$, we have $|z-\xi| \le |\beta-\beta^{\ast}| + |\xi-\xi^{\ast}|$, so Corollary~\ref{GammaArgCycle} shows that $|z-\xi| \le 2\frac{\pi}{(\log \lfloor 16\rfloor-1)^2} < 2$.
    So $(\Gamma-A)(S(\beta,\beta^\ast)) \subset B(-A(\xi),2r)$ as well.

    Let $\theta := \arg\Gamma(\beta) = \arg\Gamma(\beta^\ast)$, and let $R := |\Gamma(\xi+1)|$.
    Now we will consider
    \[
        (\Gamma-A)\big(T(\beta,R) \cup T(\beta^\ast,R)\big).
    \]
    By (\ref{eq:A(z)/4}), $-A\big(T(\beta,R) \cup T(\beta^\ast,R)\big) \subset B(-A(\xi),r)$.
    If $z \in T(\beta,R) \cup T(\beta^\ast,R)$, then $(\Gamma-A)(z) \in B(-A(\xi),r) + \Gamma(z)$. 
    We observe that as $z$ moves in $T(\beta,R)$ or $T(\beta^\ast,R)$ from $S(\beta,\beta^\ast)$ to $S(\rho(\beta,R),\rho(\beta^\ast,R))$, the set $B(-A(\xi),r)+\Gamma(z)$ moves away from 0. 
    Indeed, along $T(\beta,R)$ and $T(\beta^\ast,R)$, $|\Gamma(z)|$ strictly increases and $\arg\Gamma(z) = \arg(-A(\xi))$ is fixed.
    So $(\Gamma-A)\big(T(\beta,R) \cup T(\beta^\ast,R)\big)$ is confined to the following region 
    \[
        U := \left\{a+b : a \in B(-A(\xi),r)\wedge \arg b = \theta\wedge r \le |b| \le R\right\}.
    \]

    Next we consider $(\Gamma-A)\big(S(\rho(\beta,R),\rho(\beta^\ast,R))\big)$. 
    The image of $S\big(\rho(\beta,R),\rho(\beta^\ast,R)\big)$ under $\Gamma$ is a circle of radius $R$ centered at 0 (traversed once).
    Since
    \begin{align*}
        R 
            &= |\xi||\Gamma(\xi)|  
                &\\
            &> 12 |\Gamma(\xi)|
                &\text{by choice of $\Re(\xi) \ge 16$} \\
            &=3|A(\xi)|
                &\text{by (\ref{eq:Gamma(beta_0)})} \\
            &> 2 \left(\frac{5}{4}|A(\xi)| \right)
                &\\
            &> 2|A(z)|
                &\text{for all $z \in S\big(\rho(\beta,R),\rho(\beta^\ast,R)\big)$, by (\ref{eq:A(z)/4}}),
    \end{align*}
    the image of $S\big(\rho(\beta,R),\rho(\beta^\ast,R)\big)$ under $\Gamma-A$ is contained in the annulus
    \[
        V:=\left\{z\in\CC : \frac{R}{2} < |z| <\frac{3R}{2}\right\}.
    \]
    Since $R > 3|A(\xi)|$, $V\cap B(-A(\xi),2r)=\emptyset$. 
    See Figure~\ref{fig:region1}: $B(-A(\xi),r)$ is represented by the black disk, it is surrounded by the boundary of $B(-A(\xi),2r)$, the straight lines represent the boundary of $U$, and $V$ is the annulus determined by the two circles centered at $0$.
\begin{figure}[h]
\centering
\includegraphics[width=0.5\textwidth]{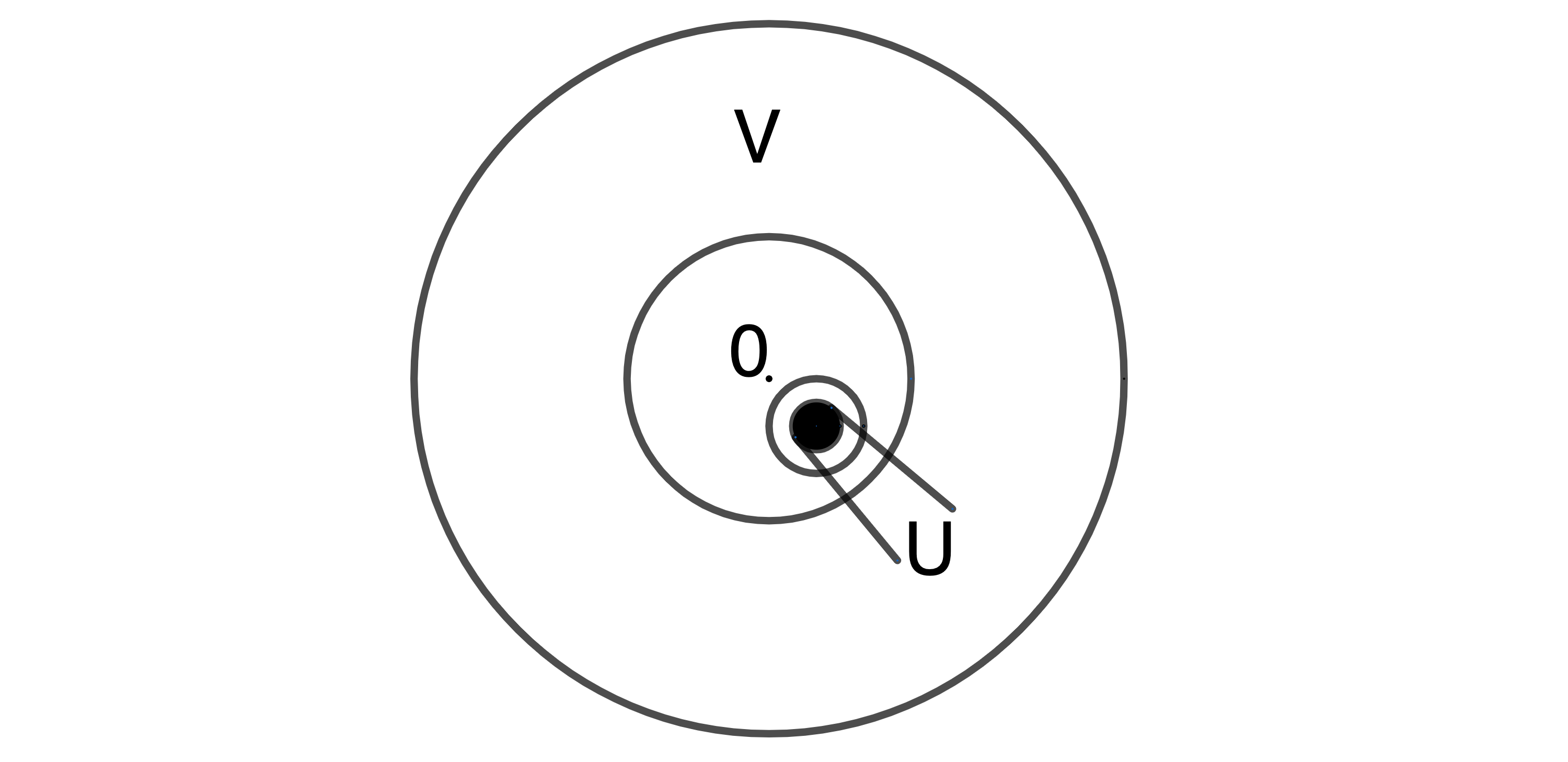}
\caption{Layout of $B(-A(\xi),r)$, $B(-A(\xi),2r)$, $U$ and $V$.}
\label{fig:region1}
\end{figure}
    Furthermore, the inequalities above show that for all $z\in S\big(\rho(\beta,R),\rho(\beta^\ast,R)\big)$, we have that $|\Gamma(z) - (\Gamma-A)(z)| = |A(z)| < \frac{R}{2}$, showing that for all $z\in S\big(\rho(\beta,R),\rho(\beta^\ast,R)\big)$, $(\Gamma-A)(z)$ is in the ball centered at $\Gamma(z)$ of radius $\frac{R}{2}$. 
    
    The image of $K(\beta,R)$ under $\Gamma-A$ must be a closed curve.
    We know that
    \begin{align*}
        (\Gamma-A)\big(S(\rho(\beta,R),\rho(\beta^\ast,R))\big)&\subset V\\
        (\Gamma-A)\big(S(\beta,\beta^\ast)\big) &\subset B(-A(\xi),2r)\\
        (\Gamma-A)\big(T(\beta,R) \cup T(\beta^\ast,R)\big) &\subset U.
    \end{align*}
    We also know that $B(-A(\xi),2r)$ and $U$ are connected sets which do not contain 0. 
    Since for $z \in S\big(\rho(\beta,R),\rho(\beta^\ast,R)\big)$, $\Gamma(z)$ parameterizes the circle of radius $R$ centered at 0 contained in $V$, and $(\Gamma-A)(z)\subset B\left(\Gamma(z),\frac{R}{2}\right)$, 
    we have that $(\Gamma-A)\big(S(\rho(\beta,R),\rho(\beta^\ast,R))\big)$ does make at least most of a turn around 0.
    
    So the winding number of $(\Gamma-A)(K(\beta,R))$ around 0 must be equal to the winding number around 0 of the curve determined by concatenating the boundaries of $B(-A(\xi),r)$ and $U$ with $\partial B\left(0,R\right)$, which is 1.

    So far we have found one zero of $\Gamma-A$. In order to find more, all we need to do is choose a new starting point $\xi'\in G$ not contained in the region bounded by $K(\beta,R)$ and repeat the construction above. 
\end{proof}

\begin{rmk}
    The diameter of the curve $K(\beta,R)$ used in the previous proof is at most $\displaystyle\frac{\pi}{(\log \lfloor \Re(\beta)\rfloor-1)^2}+1$.
\end{rmk}

\begin{cor}
    The set $\{z\in\CC : \Gamma(z) = z\}$ is Zariski dense in $\CC$. 
\end{cor}
\begin{proof}
    Follows immediately from Theorem~\ref{OneVarEC} by considering $p(X,Y):= X-Y$. 
\end{proof}

\begin{rmk}
    In fact, by doing a few more computations, one can show that, under the classical identification $\CC\cong\RR^2$, the set $\{z\in\CC : p(z,\Gamma(z)) = 0\}$ is Zariski dense in $\RR^2$, whenever $p(X,Y)$ satisfies the conditions of Theorem~\ref{OneVarEC}. 
    Indeed, by performing calculations similar to those in the proof of Lemma \ref{Limity'toInfty}, one can obtain the asymptotic slope of the curve $y(x)$ implicitly defined by $|\Gamma(x+iy)| = |A(x+iy)|$. 
    This can then be used to show that no real algebraic curve in $\RR^2$ can intersect the graph of $y(x)$ infinitely many times. 
\end{rmk}

\begin{cor}\label{DistributionLowerBd}
    Let $A(z)$ be a non-constant algebraic function.
    For any $\epsilon > 0$, there is an $R_{\epsilon} \in \RR^+$ such that for all $R > R_{\epsilon}$, there are at least $\frac{2R}{1+2\epsilon}-c$ zeroes of $\Gamma-A$ in the ball $B(0,R)$, where $c$ is a constant depending on $A$.
\end{cor}

\begin{proof}
    We will follow the notation introduced in the proof of Theorem~\ref{OneVarEC}.
    Let $N$ and $G$ be as before.
    Let $b \in G$ be a point of least modulus such that 
    \[
        |\Gamma(b)|= \frac{|A(b)|}{4}.
    \]
    Let $\displaystyle c = \frac{2|b|}{1+\frac{2\pi}{(\log\lfloor \Re(b)\rfloor-1)^2}}$.
    Now let $\epsilon > 0$, and let $x = \max\left\{16, \exp\left(\sqrt{\frac{2\pi}{\epsilon}}+1\right)\right\}$.
    Then any $z \in G$ with $\Re(z) \ge x$ satisfies $\displaystyle\frac{2\pi}{(\log \lfloor\Re(z)\rfloor -1)^2} < \epsilon$.
    Starting from $x$ we will find $\xi \in G$ with $\Re(\xi) \ge x$ satisfying 
    \begin{equation}\label{eq:Gamma(beta_0)2}
        |\Gamma(\xi)| = \frac{|A(\xi)|}{4}.
    \end{equation}
    If $|\Gamma(x)| = \frac{|A(x)|}{4}$, then let $\xi = x$.
    If $|\Gamma(x)| < \frac{|A(x)|}{4}$, then let $\xi$ be the nearest point on the horizontal line extending to the right of $x$ satisfying (\ref{eq:Gamma(beta_0)2}), which exists by Corollary~\ref{GammaLimitXtoInfty}. 
    If $|\Gamma(x)| > \frac{|A(x)|}{4}$, then let $\xi$ be the nearest point on the vertical line extending up from $x$ satisfying (\ref{eq:Gamma(beta_0)2}), which exists by Corollary~\ref{GammaLimitYtoInfty}. 
    Moreover, bounds can be computed on how far up or to the right $\xi$ can be from $x$ using equations (\ref{eq:modgammaexpdecrease}) and (\ref{eq:stirlingineq}).
    
    Let $K,\beta^\ast$ be as in the proof of Theorem~\ref{OneVarEC}.
    We will find $\xi' \in G$ with $|\xi'| >|\xi|$ and $|\xi'-\xi| < 1+2\epsilon$ such that
    \[
        |\Gamma(\xi')|= \frac{|A(\xi')|}{4}
    \]
    and the processes in the proof of Theorem~\ref{OneVarEC} starting with $\xi$ and $\xi'$ respectively give distinct zeroes of $\Gamma-A$.
    This shows that increasing the radius $R$ by $1+2\epsilon$ yields at least one zero of $\Gamma-A$ in $B(0,R+1+2\epsilon)\setminus B(0,R)$.
    
    First, our choice of $\xi$ gives $|\beta^\ast-\xi| \le |\beta^\ast - \beta| + |\beta - \xi| < 2\epsilon$.
    Note that for any $w \in G$ with $|w-\beta^\ast|<1$, we have
    \[
        |A(w)-A(\xi)| < \frac{|A(\xi)|}{4}
    \]
    i.e., $\frac{3}{4}|A(\xi)| < |A(w)| < \frac{5}{4}|A(\xi)|$.
    Since $|\Gamma(\beta^\ast)| = \frac{|A(\xi)|}{4}$, we have
    \begin{align*}
        |\Gamma(\beta^\ast-1)| 
            &= \frac{|\Gamma(\beta^\ast)|}{|\beta^\ast-1|} = \frac{|A(\xi)|}{4|\beta^\ast-1|} < \frac{\frac{3}{4}|A(\xi)|}{4} \\
        |\Gamma(\beta^\ast+1)|
            &= |\beta^\ast||\Gamma(\beta^\ast)| = \frac{|\beta^\ast||A(\xi)|}{4} > \frac{\frac{5}{4}|A(\xi)|}{4}.
    \end{align*}
    Now let $\theta = \arg \Gamma(\beta^\ast)$, and consider the component $ C \subset \{z : \arg \Gamma(z) = \theta\}$ passing through $\beta^\ast$.
    Let $w_1$ be the point on $C$ with $\Re (w_1) < \Re(\beta^\ast)$ such that $|w_1-\beta^\ast|=1$, and let $w_2$ be the point on $C$ with $\Re (w_2) > \Re(\beta^\ast)$ such that $|w_2-\beta^\ast| = 1$.
    Since the rate of change of $|\Gamma(z)|$ is larger along curves of fixed argument than along horizontal lines, $|\Gamma(w_1)| < |\Gamma(\beta^\ast-1)|$ and $|\Gamma(w_2)| > |\Gamma(\beta^\ast+1)|$.
    So
    \begin{align*}
        |\Gamma(w_1)|-\frac{|A(w_1)|}{4} &< |\Gamma(w_1)|-\frac{\frac{3}{4}|A(\xi)|}{4} < |\Gamma(w_1)| - |\Gamma(\beta^\ast-1)| < 0 \\
        |\Gamma(w_2)|-\frac{|A(w_2)|}{4} &> |\Gamma(w_2)|-\frac{\frac{5}{4}|A(\xi)|}{4} > |\Gamma(w_2)| - |\Gamma(\beta^\ast+1)| > 0
    \end{align*}
    By the intermediate value theorem, there must be a point $\xi'$ in the segment of $C$ between $w_1$ and $w_2$ such that 
    \[
        |\Gamma(\xi')| = \frac{|A(\xi')|}{4}.
    \]
    Let $K'$ be the curve and $\beta'$ the lower left corner point constructed via the process in the proof of Theorem~\ref{OneVarEC} starting with $\xi'$.
    Then $T(\beta^\ast,\rho(\beta^\ast,R))$, which is the ``top'' of $K$, runs along $C$, while $T(\beta',\rho(\beta',R'))$, which is the ``bottom'' of $K'$, lies on or above $C$.
    So the zeroes of $\Gamma-A$ in the regions bounded by $K$ and $K'$ must be distinct.

    Finally, since $\overline{\Gamma(z)} = \Gamma(\overline{z})$, reflections of the argument of Theorem~\ref{OneVarEC} and the argument above hold in the lower half quadrant of $\CC$.
    So there are actually at least two new zeroes of $\Gamma-A$ inside the ball $B(0,R+1+2\epsilon)$ compared to $B(0,R)$.
\end{proof}

A natural question in Diophantine geometry is whether the intersection of an algebraic variety $V$ with the graph of $\Gamma$ can contain a Zariski dense set of $\Gamma$-special points.\footnote{This question is analogous to the Manin--Mumford and Andr\'e--Oort conjectures, which are now major theorems in arithmetic geometry (see \cite{raynaud,mcquillan} and \cite{pila:andre-oort,pila-shankar-tsimerman,tsimerman:andre--ort}).}
As we have mentioned, we do not have a complete characterization of such points, but at least we know that all positive integers are $\Gamma$-special. 
Since $\Gamma(n)=(n-1)!$, it is a simple matter to show that the intersection of a plane algebraic curve $V$ and the graph of $\Gamma$ cannot have infinitely many integer points due to growth reasons. More generally, we have the following result.

\begin{lemma}
    Let $V\subseteq\CC^{2}$ be an algebraic curve. 
    Then the set $\{x\in\RR^+ : (x,\Gamma(x))\in V\}$ is finite.
\end{lemma}
\begin{proof}
    If $V$ is not definable over $\RR$, then $V(\RR)$ must be finite, so we only need to consider the case when $V$ is defined over $\RR$. 
    Since the restriction of $\Gamma$ to the positive reals is definable in an o-minimal expansion of $\RR$ \cite{OminGamma}, the set $\{x\in\RR^+ : (x,\Gamma(x))\in V\}$ is either finite or has positive real dimension. 
    It cannot have positive real dimension because $\Gamma$ is a transcendental function. 
\end{proof}

\subsection{Adapting the method}
\label{subsec:discussion}
As mentioned in the introduction, one of our goals is not only to develop a proof of Theorem~\ref{OneVarEC}, but also to have one that can be adapted to other functions, in particular to periodic functions. 
We explain here how to adapt the proof for $\exp$, using Rouch\'e's theorem in place of the argument principle.

\begin{prop}
\label{prop:ecexp}
    Let $p(X,Y)\in\CC[X,Y]$ be an irreducible polynomial depending on $Y$. Then the set $\{z\in\CC: p(z,\exp(z))=0\}$ is infinite.
\end{prop}
\begin{proof}
We will show that there are infinitely many solutions to the equation $\exp(z)= A(z)$, where $A(z)$ is a non-zero algebraic function. 
Choose a closed disk $D$ centered at 0 of large enough radius, so that it contains all the zeroes and poles of $A(z)$, and let $G$ be the domain obtained by removing $D$ from the domain of $A$. 
As before, there is $N\in\mathbb{R}^+$ such that for all $z\in G$ with $|z|\ge N$ and all $w \in \mathbb{C}$ with $|w|<13$,
    \begin{equation*}
        |A(z+w)-A(z)| < \frac{|A(z)|}{4}.
    \end{equation*}

Choose $\xi = x_1+iy_1\in G$, with $|\xi|>N$, such that $|\exp(\xi)| = \exp(x_1) = \frac{|A(\xi)|}{4}$. 
In particular $A(\xi)\neq 0$. 
Given $z=x+iy\in\mathbb{C}$ and a positive integer $n$, we define $R(z)$ to the be positively oriented rectangle with corners $x+iy, x+2+iy, x+i(y+2\pi)$ and $x+2+i(y+2\pi)$. 
Consider now the rectangle $R(\xi)$ and
choose $\beta=x_1+iy_2$ on the left vertical side of $R(\xi)$ so that $\arg(\exp(\beta)) = \arg(-A(\xi))$. 
We make the following observations.
\begin{enumerate}[(a)]
    \item For every $z\in R(\beta)$ we have that $|z-\xi|< \sqrt{16\pi^2 +4} <13$, so in particular $|A(z)-A(\xi)| < \frac{|A(\xi)|}{4}$ and $|A(z)| \leq |A(z)-A(\xi)| + |A(\xi)| <\frac{5}{4}|A(\xi)|$.
    \item For every $z$ on the left vertical side of $R(\beta)$ we have 
\begin{equation*}
    |\exp(z)-A(\xi)| \geq |\exp(x_1) - |A(\xi)|| = \frac{3}{4}|A(\xi)|.
\end{equation*}
\item For every $z$ on one of the horizontal sides of $R(\beta)$ we have $\arg(\exp(z)) = \arg(-A(\xi))$, so the points $\exp(z)$, $A(\xi)$ and $0$ are collinear and distinct (since $A(\xi)\neq0$), with 0 being in between the other two. 
Therefore $|\exp(z)-A(\xi)| > |A(\xi)|$.
\item For every $z$ on the right vertical side of $R(\beta)$ we have 
    \begin{equation*}
        |\exp(z)|= |\exp(x_1+2+iy)|  
            = \exp(2) |\exp(\xi)|
            > 4|\exp(\xi)|
            =|A(\xi)|.
    \end{equation*}
\end{enumerate}
Overall, we have shown that for every $z\in R(\beta)$ we have  
\begin{equation*}
    |\exp(z)-A(\xi)| > \frac{|A(\xi)|}{4}\geq |A(z)-A(\xi)|.
\end{equation*}
So by Rouch\'e's theorem we conclude that the number of zeroes of $\exp(z)-A(z)$ inside $R(\beta)$ equals the number of zeroes of $\exp(z)-A(\xi)$ inside of $R(\beta)$, which is 1 (since $A(\xi)$ is non-zero). 
Repeating the argument for different choices of $\xi$ gives the infinitely many solutions.
\end{proof}

Combined with the methods of the following section, this proof can be adapted to obtain an analogue of Theorem \ref{thm:ecforgamma} for $\exp$, thus giving yet a new proof of \cite[Proposition 2]{brownawell-masser} (see also \cite[Theorem 1.5]{expEC} for a different proof).

\section{Proof of the Main Result}
\label{sec:proof}

Let $V\subseteq\CC^{2n}$ be an algebraic variety satisfying the hypotheses of Theorem~\ref{thm:ecforgamma}. 
Standard arguments about intersecting $V$ with generic hyperplanes allow us to reduce to the case where $\dim V=n$, see e.g.~\cite[Lemma 4.30]{vahagn:adequate} and \cite[Proposition 4.4]{genericsol}. 
To prove Theorem \ref{thm:ecforgamma} it suffices then to solve systems of equations of the form
\begin{equation}
\label{eq:system}
\left\{\begin{array}{ccc}
     \Gamma(z_1) &=& A_1(z_1,\ldots,z_n)\\
    &\vdots&\\
    \Gamma(z_n)&=& A_n(z_1,\ldots,z_n)
\end{array}
    \right\},
\end{equation}
where $A_1,\ldots,A_n$ are algebraic functions determined by $V$. 
We will be interested in the behavior of these functions as the $z_i$ tend to $\infty$, and some care is needed to do this correctly. 

\subsection{The general set-up}
The following constructions are mostly folklore, but for the sake of precision, we follow \cite[\textsection 3]{expEC}.

Consider the complex projective space $\mathbb{P}^n$. Given $\ell\in\{0,\ldots,n\}$ define the chart $U_\ell:=\{[z_0:\cdots:z_n]\in\mathbb{P}^n : z_\ell=1\}$. 
We now fix the embedding $\CC^n\hookrightarrow\mathbb{P}^n$ by identifying $\CC^n$ with $U_0$ via the usual map $(z_1,\ldots,z_n)\mapsto[1:z_1:\cdots:z_n]$. 

Given $\ell\in\{1,\ldots,n\}$, a point $\mathbf{c}:=[0:c_1:\cdots:c_n]\in U_\ell\subset\mathbb{P}^n$ (where $c_\ell=1$), and $\epsilon>0$, we define the \emph{polydisk of radius $\epsilon$ centered at $\mathbf{c}$} to be
\begin{equation*}
    D^{\ast}(\mathbf{c},\epsilon):=\{[z_0:\cdots:z_n]\in\mathbb{P}^n : |z_0|<\epsilon\wedge |z_k-c_k|<\epsilon \mbox{ for } k\in\{1,\ldots,n\}\wedge z_\ell=1\}.
\end{equation*}
The restriction of this polydisk to $U_0$ gives
\begin{equation*}
    D(\mathbf{c},\epsilon):=\left\{(z_1,\ldots,z_n)\in\CC^n : |z_\ell|>\frac{1}{\epsilon}\wedge \left|\frac{z_k}{z_\ell} - c_k\right|<\epsilon \mbox{ for }k\in\{1,\ldots,n\}\right\}.
\end{equation*}
Observe that for $\mathbf{z}\in D(\mathbf{c},\epsilon)$ we have that $z_k\in B(c_kz_\ell,\epsilon|z_\ell|)$, for every $k\in\{1,\ldots,n\}$. Given $\theta\in\RR$ and $\eta\in(\theta,\theta+2\pi]$ we define
\begin{equation*}
    D_{\theta,\eta}(\mathbf{c},\epsilon):=\{\mathbf{z}\in D(\mathbf{c},\epsilon) : \theta<\arg(z_\ell)<\eta \}.
\end{equation*}
Since we will be working only on $Q_{\alpha,0}^n$, it suffices to take $\arg(z)$ to always be in $\left(0, \frac{\pi}{2}\right)$. 
We also define $D^\ast_{\theta,\eta}(\mathbf{c},\epsilon):=D_{\theta,\eta}(\mathbf{c},\epsilon)\cup\left(D^{\ast}(\mathbf{c},\epsilon)\setminus D(\mathbf{c},\epsilon)\right)$.

Then by \cite[Proposition 3.2]{expEC} we know that for the variety $V$ there is a Zariski open dense subset $P\subset\mathbb{P}^n\setminus U_0$ with the following property: for every $\ell\geq 1$, $\mathbf{c}\in P\cap U_\ell$, $\theta\in\RR$, $\eta\in(\theta,\theta+2\pi]$ and all sufficiently small polydisks $D^{\ast}(\mathbf{c},\epsilon)$ there are unique continuous maps $A_1,\ldots,A_n:D^{\ast}_{\theta,\eta}(\mathbf{c},\epsilon)\to\mathbb{P}^1$ such that
\begin{enumerate}[(i)]
    \item the restriction of each $A_i$ to $D_{\theta,\eta}(\mathbf{c},\epsilon)$ is complex analytic, and
    \item for all $\mathbf{z}\in D_{\theta,\eta}(\mathbf{c},\epsilon)$, $(\mathbf{z}, A_1(\mathbf{z}),\ldots,A_n(\mathbf{z}))\in V$.
\end{enumerate}
\begin{rmk}
    Choose $\ell\in\{1,\ldots,n\}$. Observe that for every $\epsilon>0$ there are $\epsilon',\delta>0$ such that for all $\mathbf{c},\mathbf{c}'\in C\cap U_\ell$, if $\|\mathbf{c}-\mathbf{c}'\|<\delta$ then $D(\mathbf{c}',\epsilon')\subseteq D(\mathbf{c},\epsilon)$.
\end{rmk}

Choose $0<\epsilon<\frac{1}{2}$ and choose $\mathbf{c}$ so that $c_i\geq 1$ for all $i\in{1,\ldots,n-1}$. 
Let $A_1,\ldots,A_n$ be the corresponding algebraic maps defined on $D^{\ast}_{\theta,\eta}(\mathbf{c},\epsilon)$, where $\theta\in\left(-\pi,-\frac{\pi}{2}\right)$ and $\eta\in\left(\frac{\pi}{2},\pi\right)$. 
Since $V$ is free and $P$ is Zariski dense, we may assume that $A_i(c_1z,\ldots,c_{n-1}z,z)$ is not identically zero nor identically $\infty$, for all $i\in\{1,\ldots,n\}$. 

For the rest of this section we set $\ell=n$, and $A_1,\ldots,A_n$, $\mathbf{c}$, $\epsilon$ and $D_{\theta,\eta}(\mathbf{c},\epsilon)$ will remain fixed. 
Given $z_n$ with $|z_n|>\frac{1}{\epsilon}$ we have that for every $z_i\in B(c_iz_n,\epsilon|z_n|)$, $|z_i|<(c_i+\epsilon)|z_n|$. 
This shows that the coordinates of the points in $D(\mathbf{c},\epsilon)$ grow asymptotically like $|z_n|$.
So there is $M>\frac{1}{\epsilon}$ such that for every $i\in\{1,\ldots,n\}$ there are $a_i,b_i\in\RR^+$ and $d_i\in\mathbb{Z}$ such that for all $\mathbf{z}\in D_{\theta,\eta}(\mathbf{c},\epsilon)$ with $|z_n|>M$ we have
\begin{equation}
\label{eq:algfunasympt}
    a_i|z_n|^{d_i-1} < |A_i(\mathbf{z})|\leq b_i|z_n|^{d_i}. 
\end{equation}

\subsection{The proofs}

Before giving the proof of the main theorem, we present a result that is stronger in some aspects (as it allows for equations between $\Gamma$ and other holomorphic functions instead of only algebraic functions) and weaker in others (as it requires the holomorphic functions to have a non-zero finite limit at infinity). 
\begin{prop}
\label{prop:analyticec}
    Let $f_1,\ldots,f_n:Q_{\alpha,0}^n\to\CC$ be holomorphic functions such that for every $k\in\{1,\ldots,n\}$ there is $w_k\in\CC^\times$ such that
    \begin{equation*}
        \lim_{\mathbf{z}\to\infty}f_k(\mathbf{z}) = w_k.
    \end{equation*}
    Then there are infinitely many solutions to the system of equations
    \begin{align*}
            \Gamma(z_1) & =  f_1(z_1,\ldots,z_n)  \\
            & \vdots \\
            \Gamma(z_n) & =  f_n(z_1,\ldots,z_n).
        \end{align*}
\end{prop}
\begin{proof}
    For every $k\in\{1,\ldots,n\}$ choose positive real numbers $0<r_k<R_k$ such that $r_k<|w_k| < R_k$. 
    Let $\delta:=\min_{k=1,\ldots,n}\{R_k-|w_k|, |w_k|-r_k\}$, and let $M>0$ be such that if $\mathbf{z}\in Q_{\alpha,0}^n$ satisfies $\|\mathbf{z}\|>M$, then  for every $k\in\{1,\ldots,n\}$ we have $|f_k(\mathbf{z})-w_k| < \delta$. 
    Choose $\boldsymbol{\beta}\in Q_{\alpha,0}^n$ so that $\|\boldsymbol{\beta}\|\geq M$, $|\Gamma(\beta_k)|=r_k$ and $\arg(\Gamma(\beta_k)) = \arg(-w_k)$. 
    Let $\Omega_k$ be the region bounded by $K(\beta_k,R_k)$, and let $\Omega := \Omega_1\times\cdots\times\Omega_n$.
    We remark that for every $\mathbf{z}\in\partial\Omega$ there is $k\in\{1,\ldots,n\}$ such that $z_k\in K(\beta_k,R_k)$. 
    
    Now, for every $z\in K(\beta_k,R_k)$ one of the following scenarios must occur:
    \begin{enumerate}[(a)]
        \item $|\Gamma(z)| = r_k$, in which case $|\Gamma(z)-w_k|\geq |w_k|-r_k$.
        \item $\arg(\Gamma(z))=\arg(\Gamma(\beta_k))=\arg(-w_k)$, in which case the line through $\Gamma(z)$ and $w_k$ goes through 0, and 0 is in between $\Gamma(z)$ and $w_k$. Then $|\Gamma(z)-w_k|$ is the length of this line segment, which is at least $|w_k| + r_k$.
        \item $|\Gamma(z)| = R_k$, in which case $|\Gamma(z)-w_k|\geq R_k-|w_k|$.
    \end{enumerate}
    
    So given $\mathbf{z}\in\partial\Omega$ there is $k\in\{1,\ldots,n\}$ such that $|\Gamma(z_k)-w_k|>\delta$, and thus $|\Gamma(z_k)-w_k| > |f_k(\mathbf{z})-w_k|$.
    By the multivariable version of Rouch\'e's theorem (see e.g.~\cite[Theorem 2 in Chapter IV, \textsection 18.55]{shabat}) we know that the number of zeros of $(\Gamma(z_1)-f_1(\mathbf{z}),\ldots,\Gamma(z_n)-f_n(\mathbf{z}))$ in $\Omega$ equals the number of zeros of $(\Gamma(z_1)-w_1,\ldots,\Gamma(z_n)-w_n)$ in $\Omega$. 
    By the construction of $K(\beta_k,R_k)$, $\Gamma(\Omega_k) = \{z\in\CC : r_k<|z|<R_k\}$, so there is $z\in \Omega_k$ such that $\Gamma(z)=w_k$, and thus there is a zero of $(\Gamma(z_1)-w_1,\ldots,\Gamma(z_n)-w_n)$ in $\Omega$. 
    We may now repeat the argument by increasing $\|\boldsymbol{\beta}\|$ to get infinitely many zeros of $(\Gamma(z_1)-f_1(\mathbf{z}),\ldots,\Gamma(z_n)-f_n(\mathbf{z}))$. 
\end{proof}

Now we move towards the proof of Theorem~\ref{thm:ecforgamma}. We need one lemma first. 

\begin{lemma}
\label{lem:ineqalgfunc}
    Let $A:D_{\theta,\eta}(\mathbf{c},\epsilon)\to\CC$ be an algebraic function which is continuous on $D^{\ast}_{\theta,\eta}(\mathbf{c},\epsilon)$.
    Let $d \in \RR^+$. 
    Then there is $N \in \RR^+$ such that for all $\mathbf{z} \in D_{\theta,\eta}(\mathbf{c},\epsilon)$ with $\|\mathbf{z}\|>N$ and $\mathbf{w} \in \CC^n$ with $\|\mathbf{w}\| \le d$ satisfying $\mathbf{z}+\mathbf{w}\in D_{\theta,\eta}(\mathbf{c},\epsilon)$, 
    \begin{align*}
        |A(\mathbf{z}+\mathbf{w})-A(\mathbf{z})| < \frac{|A(\mathbf{z})|}{4} .
    \end{align*}
\end{lemma}
\begin{proof}
Given $\mathbf{w} \in \CC^n$, then by continuity of $A$ we have that
\begin{equation*}
    \lim_{\|\mathbf{z}\|\to+\infty}\left|\frac{A(\mathbf{z}+\mathbf{w})}{A(\mathbf{z})} - 1\right| = 0,
\end{equation*}
where $\mathbf{z}$ only moves within $D_{\theta,\eta}(\mathbf{c},\epsilon)$. If $\mathbf{w}$ is restricted to move within a compact set, then the limit is uniform, so there exists $N$ satisfying the conclusion of the lemma.
\end{proof}

\begin{proof}[Proof of Theorem~\ref{thm:ecforgamma}]
Choose $M_1>0$ and $j\in\{1,\ldots,n\}$ such that for all $i\in\{1,\ldots,n\}$ and all $\mathbf{z}\in D_{\theta,\eta}(\mathbf{c},\epsilon)$ with $|z_n|>M_1$, $|A_j(\mathbf{z})|\leq |A_{i}(\mathbf{z})|$. 
Let $r = \max(16, d_i-d_j+3)$.
Choose $\boldsymbol{\xi}\in D_{\theta,\eta}(\mathbf{c},\epsilon)$ with $|\xi_n|>M_1$ satisfying the following conditions:
\begin{enumerate}[(i)]
    \item  For all $i\in\{1,\ldots,n\}$, $\xi_{i}\in Q_{r,r}$ and 
    \begin{equation*}
        |\xi_i|\geq \frac{5b_i}{a_j(c_i-\epsilon)^{d_i-d_j+1}}.
    \end{equation*}
    \item For all $\mathbf{z}\in \prod_{i=1}^{n}B(\xi_{i},d_i-d_j+3)$ (here we use Lemma~\ref{lem:ineqalgfunc})
\begin{equation*}
    |A_i(\mathbf{z}) - A_i(\boldsymbol{\xi})| < \frac{|A_{i}(\boldsymbol{\xi})|}{4}.
\end{equation*}
We remark that by the choice of $A_j$ we always have that $d_i-d_j\geq 0$.
    \item For all $i\in\{1,\ldots,n\}$, $\displaystyle|\Gamma(\xi_i)| = \frac{|A_j(\boldsymbol{\xi})|}{4}$.
\end{enumerate}

Choose $\beta_i\in S(\xi_i,\xi_i^\ast)$ so that  $\arg(\Gamma(\beta_i))=\arg (-A_i(\boldsymbol{\xi}))$. 
Given $i\in\{1,\ldots,n\}$ consider the curve $K(\beta_i, R_i)$, where $R_i = |\Gamma(\beta_i+d_i-d_j+2)|$. 
Let $\Omega_i\subset\mathbb{C}$ be the region bounded by $K(\beta_i,R_i)$. 
Define $\Omega:=\Omega_1\times\cdots\times\Omega_n$.

For every $\mathbf{z}\in\partial\Omega$ there is $k\in\{1,\ldots,n\}$ such that $z_k\in K(\beta_k,R_k)$. 
Observe that in this case we have
\begin{enumerate}[(a)]
    \item If $z_k\in S(\beta_k,\beta_k^\ast)$, then $|\Gamma(z_k)| = |\Gamma(\beta_k)| = |\Gamma(\xi_k)|= \frac{|A_j(\boldsymbol{\xi})|}{4}\leq \frac{|A_k(\boldsymbol{\xi})|}{4}$, so
    \begin{equation*}
        |\Gamma(z_k)-A_k(\boldsymbol{\xi})| \geq\frac{3}{4}|A_k(\boldsymbol{\xi})|.
    \end{equation*}
    \item If $z_k\in T(\beta_k,R_k) \cup T(\beta_k^\ast,R_k)$, then the points $\Gamma(z_k)$, $A_k(\boldsymbol{\xi})$ and $0$ are collinear, with 0 being in between the other two. Therefore 
    \begin{equation*}
        |\Gamma(z_k)-A_k(\boldsymbol{\xi})| \geq|A_k(\boldsymbol{\xi})|.
    \end{equation*}
    \item If $z_k\in S\big(\rho(\beta_k,R_k),\rho(\beta_k^\ast,R_k)\big)$, then we get
    \begin{align*}
        |\Gamma(z_k)| &= R_k & \\
        &= |\Gamma(\beta_k)|\prod_{m=0}^{d_k-d_j+1}|\beta_k+m| & \mbox{by (\ref{eq:diffeq1})}\\
        &\geq |\Gamma(\xi_k)||\xi_k|^{d_k-d_j+2} & \mbox{since } |\beta_k|\geq|\xi_k|\\
        &\geq \frac{|A_j(\boldsymbol{\xi})|}{4} \cdot \frac{5b_k}{a_j(c_k-\epsilon)^{d_k-d_j+1}}\cdot |\xi_k|^{d_k-d_j+1} & \mbox{by (i)}\\
        &\geq \frac{5b_k}{4a_j}|A_j(\boldsymbol{\xi})||\xi_n|^{d_k-d_j+1} & \mbox{since }\boldsymbol{\xi}\in D(\mathbf{c},\epsilon)\\
        &\geq \frac{5b_k}{4a_j}a_j|\xi_n|^{d_j-1}|\xi_n|^{d_k-d_j+1} & \mbox{by (\ref{eq:algfunasympt})}\\
        &= \frac{5}{4}b_k|\xi_n|^{d_k} & \\
        &\geq \frac{5}{4}|A_k(\boldsymbol{\xi})| & \mbox{by (\ref{eq:algfunasympt})}.
    \end{align*}
    Therefore $|\Gamma(z_k)-A_k(\boldsymbol{\xi})| \geq\frac{1}{4}|A_k(\boldsymbol{\xi})|$.
\end{enumerate}
This shows that for all $\mathbf{z}\in\partial\Omega$ with $z_k\in K(\beta_k,R_k)$, we have that 
\begin{equation*}
    |\Gamma(z_k)-A_k(\boldsymbol{\xi})| \geq\frac{1}{4}|A_k(\boldsymbol{\xi})|  
    > |A_k(\mathbf{z})-A_{k}(\boldsymbol{\xi})|,
\end{equation*}
by (ii). Observe that the image of $K(\beta_k,R_k)$ under $\Gamma$ is the annulus $\left\{z\in\CC : \frac{1}{4}|A_j(\boldsymbol{\xi})|\le|z|\le R_k\right\}$, so there exists $\boldsymbol{\chi}\in\Omega$ such that $\Gamma(\boldsymbol{\chi}) = (A_1(\boldsymbol{\xi}),\ldots,A_n(\boldsymbol{\xi}))$. Therefore by the multivariable version of Rouch\'e's theorem we have that there is a point $\mathbf{w}\in\Omega$ such that $\Gamma(\mathbf{w}) = (A_1(\mathbf{w}),\ldots,A_{n}(\mathbf{w}))$. 

So far we have proven that the intersection of $V\subset\CC^{2n}$ with the graph of $\Gamma$ is infinite for every positive integer $n$ and for every variety $V$ with dominant projection onto the first $n$ coordinates and no constant coordinates. 
Rabinowitsch's trick implies that this is enough to obtain Zariski density (see \cite[Proposition 4.34]{vahagn:adequate} and \cite[Lemma 3.8]{genericsol}). 
Indeed, given a proper subvariety $Z$ of $V$ 
we can find a polynomial $F\in\CC[\mathbf{X},\mathbf{Y}]$ which vanishes on $Z$ but not on all of $V$.
We can now define the following subvariety of $\CC^{2n+2}$:
\[V':=\{(\mathbf{x},x_{n+1},\mathbf{y},y_{n+1})\in\CC^{2n+2} : y_{n+1}F(\mathbf{x},\mathbf{y}) = 1\}.\]
Observe that $V'$ has a dominant projection onto the first $n+1$ coordinates, and that the image of $V'$ under the algebraic map $\CC^{2n+2}\to\CC^{2n}$ given by $(\mathbf{x},x_{n+1},\mathbf{y},y_{n+1})\mapsto(\mathbf{x},\mathbf{y})$ lands in $V\setminus Z$. 
Using what we have proven, we know that $V'$ contains a point in the graph of $\Gamma$, and therefore so does $V\setminus Z$.
\end{proof}

\begin{rmk}
    Every point in the intersection between $V$ and the graph of $\Gamma$ is a solution of a system of the form (\ref{eq:system}). There are only finitely many such systems for each $V$. For each system we can choose a corresponding domain $D_{\theta,\eta}(\mathbf{c},\epsilon)$. We can cover $\left(\mathbb{Z}^+\right)^n$ by finitely many such domains. By (\ref{eq:algfunasympt}) we can conclude that when $V$ satisfies the conditions of Theorem~\ref{thm:ecforgamma} and $\dim V=n$, then the set $\{(\mathbf{z},\Gamma(\mathbf{z}))\in V : \mathbf{z}\in \left(\mathbb{Z}^+\right)^n\}$ is finite, since when $\|\mathbf{z}\|$ is large enough, there is some $k\in\{1,\ldots,n\}$ for which $|\Gamma(z_k)|$ is much bigger than $|A_k(\mathbf{z})|$.   
\end{rmk}

Given a function $f$ and a positive integer $n$ we write $f^{\circ n}$ to denote the composition of $f$ with itself $n$ times. That is $f^{\circ 1} := f$ and $f^{\circ (m+1)} := f^{\circ m}\circ f$. Theorem \ref{thm:ecforgamma} gives a new proof of the following result, which can also be deduced from the general result that mermomorphic transcendental functions have periodic points of every period, see \cite[Theorem 2]{bergweiler}. 

\begin{cor}
    $\Gamma$ has infinitely many periodic points of every period. In other words, for every positive integer $n$, there are infinitely many points $z$ such that $\Gamma^{\circ n}(z) = z$ and $\Gamma^{\circ m}(z)\neq z$ for all $m\in\{1,\ldots,n-1\}$. 
\end{cor}
\begin{proof}
    Consider the variety
    \begin{equation*}
        V:=\left\{\begin{array}{rcc}
    X_{2} & = & Y_{1} \\
    X_{3} & = & Y_{2}\\
    &\vdots&\\
    X_1 &=& Y_n
\end{array}\right\}.
    \end{equation*}
    By Theorem~\ref{thm:ecforgamma} we know that this variety has a Zariski dense intersection with the graph of $\Gamma$. This proves that there are infinitely many points $z$ satisfying $\Gamma^{\circ n}(z) = z$. If $m\in\{1,\ldots,n-1\}$, then the points which satisfy $\Gamma^{\circ m}(z)= z$ and $\Gamma^{\circ n}(z)= z$ lie in a proper subvariety of $V$, namely the one cut out by the extra condition $X_1 = Y_m$. Since there are finitely many possible values of $m$, then by Zariski density we may find a point with the desired properties. 
\end{proof}

\begin{rmk}
    The same argument used to prove Corollary~\ref{DistributionLowerBd} shows that given $\boldsymbol{\xi}$ as in the proof of Theorem~\ref{thm:ecforgamma}, there are $n$ other nearby starting points for the argument which lead to distinct points of the form $(\boldsymbol{w},\Gamma(\boldsymbol{w}))$ in the variety $V$.
    More precisely, there are 
    $\boldsymbol{\xi}_1',\dots,\boldsymbol{\xi}_n' \in D_{\theta,\eta}(\boldsymbol{c},\epsilon)$ with $\|\boldsymbol{\xi}_i'\| > \|\boldsymbol{\xi}\|$ and $\|\boldsymbol{\xi}_i'-\boldsymbol{\xi}\| < 1 + \frac{4\pi}{(\log \lfloor \Re(\xi_i)\rfloor -1)^2}$ for $i=1,\dots,n$ which satisfy conditions (i), (ii), and (iii) in the proof of Theorem~\ref{thm:ecforgamma} and produce disjoint regions $\Omega_1',\dots,\Omega_n' \subset D_{\theta,\eta}(\boldsymbol{c},\epsilon)$, each of which contains a point $\boldsymbol{w}_i$ with $\Gamma(\boldsymbol{w}_i) = \big(A_1(\boldsymbol{w}_i),\dots,A_n(\boldsymbol{w}_i)\big)$.
    This observation could help to obtain a distribution result like Corollary~\ref{DistributionLowerBd} in the higher dimensional cases treated in Theorem~\ref{thm:ecforgamma}.
\end{rmk}

\subsubsection*{No data is associated with this publication}

\AtNextBibliography{\small}
\printbibliography

\end{document}